\documentclass[11pt,reqno]{amsart}
\usepackage{amsmath,amsfonts,amssymb,amsthm,fancyhdr,bm,mathtools,enumitem,ifthen}
\usepackage[hyphens]{url}
\usepackage[hidelinks]{hyperref}

\usepackage[margin=1in]{geometry}

\usepackage[dvipsnames]{xcolor}
\usepackage[capitalize]{cleveref}
\usepackage[color=Purple!50!white]{todonotes}
\setuptodonotes{size=\scriptsize}
\usepackage[numbers,sort]{natbib}
\usepackage{tikz,ifthen,subcaption}

\usetikzlibrary{decorations.markings,arrows.meta}
\tikzset{vert/.style={draw, fill=black, circle, inner sep=1pt}}
\pgfkeys{/tikz/.cd,
arpos/.store in=\arpos,
arpos=0.6}
\pgfkeys{/tikz/.cd,
arlen/.store in=\arlen,
arlen=2pt}
\tikzset{ar/.style={decoration={markings,mark=at position \arpos with {\arrow{Straight Barb[length=\arlen]}}}, postaction=decorate}}

\newtheorem{theorem}{Theorem}[section]
\newtheorem{lemma}[theorem]{Lemma}
\newtheorem{proposition}[theorem]{Proposition}
\newtheorem{corollary}[theorem]{Corollary}

\newtheorem{question}[theorem]{Question}
\newtheorem{problem}[theorem]{Problem}
\newtheorem{claim}[theorem]{Claim}
\theoremstyle{definition}

\crefname{equation}{equation}{equations}
\crefname{lemma}{Lemma}{Lemmas}
\crefname{proposition}{Proposition}{Propositions}
\crefname{claim}{Claim}{Claims}
\crefname{theorem}{Theorem}{Theorems}
\crefname{conjecture}{Conjecture}{Conjectures}
\crefname{figure}{Figure}{Figures}

\newlist{lemenum}{enumerate}{1}
\setlist[lemenum]{label=(\alph*), ref=\thelemma(\alph*)}
\crefalias{lemenumi}{lemma}

\AddToHook{env/lemma/begin}{\crefalias{theorem}{lemma}}
\AddToHook{env/proposition/begin}{\crefalias{theorem}{proposition}}
\AddToHook{env/corollary/begin}{\crefalias{theorem}{corollary}}
\AddToHook{env/conjecture/begin}{\crefalias{theorem}{conjecture}}
\AddToHook{env/claim/begin}{\crefalias{theorem}{claim}}
\AddToHook{env/question/begin}{\crefalias{theorem}{question}} 
\AddToHook{env/problem/begin}{\crefalias{theorem}{problem}} 
\AddToHook{env/definition/begin}{\crefalias{theorem}{definition}} 
\AddToHook{env/remark/begin}{\crefalias{theorem}{remark}}

\newcommand\ol[1]{\overline{#1}}

\newcommand\wt[1]{\widetilde{#1}}
\newcommand\ab[1]{\lvert#1\rvert}

\newcommand{\flo}[1]{\lfloor #1 \rfloor}

\newcommand{\R}{\mathbb{R}}
\newcommand{\N}{\mathbb{N}}

\let\leq\leqslant
\let\geq\geqslant

\title{Color-avoiding directed paths in tournaments}
\author{Jacob Fox}
\thanks{JF: Department of Mathematics, Stanford University, USA. Email: \texttt{jacobfox@stanford.edu}. Research supported by NSF awards DMS-2452737 and DMS-2154129}
\author{Benny Sudakov}
\thanks{BS: Department of Mathematics, ETH Z\"urich, Switzerland. Email: \texttt{benjamin.sudakov@math.ethz.ch}. Research supported by SNSF grant 200021-228014}
\author{Yuval Wigderson} 
\thanks{YW: Institute for Theoretical Studies, ETH Z\"urich, Switzerland. Email: \texttt{yuval.wigderson@eth-its.ethz.ch}. Research supported by Dr.\ Max R\"ossler, the Walter Haefner Foundation, and the ETH Z\"urich Foundation}
\date{}
\begin{document}

\begin{abstract}
We study the following Ramsey-theoretic question: given a $q$-coloring of the edges of a tournament, how long of a directed path can we guarantee whose edges avoid one of the colors? 
Questions of this type have applications in many areas, such as vector sequences, convex geometry, and extremal hypergraph theory, and have been extensively studied over the past 50 years.

We prove that if $\varepsilon>0$ is fixed and $q$ is sufficiently large, then every $q$-edge-colored $N$-vertex tournament contains a color-avoiding directed path of length $N^{1-\varepsilon}$. 
This answers a question of Gowers and Long, strengthens several of their results, and extends earlier work of Loh.
\end{abstract}
\maketitle

\section{Introduction}
\subsection{Background and main results}
Questions about directed paths in tournaments are as old as the study of tournaments itself; the very first result on tournaments is due to R\'edei \cite{Redei}, who proved that every tournament contains a Hamiltonian directed path, i.e.\ a directed path visiting every vertex. 
A substantial generalization of this result was obtained independently by Gallai, Hasse, Roy, and Vitaver \cite{MR233733,MR179105,MR225683,MR145509}, who proved that every edge-orientation of a graph $G$ contains a directed path of length at least $\chi(G)$; R\'edei's theorem then corresponds to the case where $G$ is a complete graph. Throughout this paper, by the \emph{length} of a path, we mean the number of vertices it comprises.

Our focus in this paper will be on Ramsey-theoretic questions in tournaments, wherein we are given a tournament whose edges have been assigned one of $q$ possible colors, and we wish to understand which sorts of substructures must appear in this coloring. The study of such questions can be traced back to work of Stearns \cite{MR109087} and Erd\H os--Moser \cite{MR168494} from the late 1950s, and they have become increasingly well-studied in the past few years, e.g.\ \cite{MR3980089,MR4819947,MR416967,MR3567530,MR4103733}.

One early result in this direction, due independently to Chv\'atal \cite{MR309785} and Gy\'arf\'as--Lehel \cite{MR335354}, again concerns directed paths; they proved that in every $q$-edge-coloring of an $N$-vertex tournament, there must be a monochromatic directed path of length at least $N^{1/q}$. This fact is a simple corollary of the Gallai--Hasse--Roy--Vitaver theorem mentioned above: viewing each color class as an $N$-vertex graph, these $q$ graphs have $K_N$ as their union, hence one of them must have chromatic number at least $N^{1/q}$, from which the result follows. By using a simple product construction (see \cref{prop:lex product} for details), one can show that this bound is tight whenever $N$ is a power of $q$, even when restricting to transitive tournaments.

Despite its simple proof, even very special cases of the Chv\'atal--Gy\'arf\'as--Lehel theorem have a number of remarkable consequences. For example, restricting our attention to the case $q=2$, and focusing only on \emph{transitive} tournaments, the result we obtain is equivalent to the famous Erd\H os--Szekeres lemma \cite{MR1556929} on monotone subsequences, stating that every sequence of $N$ real numbers has a monotone subsequence of length at least $\sqrt N$. Indeed, any such sequence yields a two-coloring of a transitive $N$-vertex tournament, coloring an edge red if its right endpoint is greater than its left (and blue otherwise); it is easy to see that monotone subsequences correspond to monochromatic directed paths in this coloring, from which one deduces the Erd\H os--Szekeres lemma. Similarly, if we increase $q$ but still restrict to transitive tournaments, we readily recover the multidimensional analogue of the Erd\H os--Szekeres lemma. Further generalizations of these statements and arguments exhibit deep connections to discrete geometry \cite{MR1556929} and high-dimensional partitions \cite{MR3228450}, among others.

In this paper, we study a close cousin of these classical questions, which turns out to be surprisingly rich and difficult. 
We move from the classical Ramsey-theoretic goal of finding a monochromatic structure, to the alternative goal of finding a \emph{color-avoiding} structure; that is, one which receives at most $q-1$ colors out of the palette of $q$ colors used in total. 
The study of such questions goes back at least to the seminal work of Erd\H os and Szemer\'edi \cite{MR325446}, and is a special case of the ``generalized Ramsey theory'' introduced by Erd\H os and Gy\'arf\'as \cite{MR1645678} and intensively studied since, e.g.\ \cite{MR4975150,MR1899117,MR4563218}.

The most basic question of the kind we study was 
introduced by Loh \cite{1505.07312}, who asked for the length of the longest color-avoiding directed path that is guaranteed in every $3$-edge-coloring of an $N$-vertex tournament. We denote this quantity by $g_{3,2}(N)$. More generally, for integers $q \geq r \geq 1$, let us define $g_{q,r}(N)$ analogously: it is the minimum, over all $q$-edge-colored $N$-vertex tournaments, of the number of vertices in the longest directed path whose edges receive at most $r$ colors. While our focus will be on the color-avoiding case, in which $r=q-1$, the general function is equally natural to study.

In introducing this question, Loh observed two simple upper and lower bounds on $g_{3,2}(N)$. First, we have the trivial lower bound $g_{3,2}(N) \geq \sqrt N$, 
obtained by simply merging two of the colors and applying the Chv\'atal--Gy\'arf\'as--Lehel theorem to find a monochromatic directed path in the modified coloring; any such path must receive at most two colors in the original coloring, and hence is color-avoiding. On the other hand, Loh noted that the same simple product construction witnessing the tightness of the Chv\'atal--Gy\'arf\'as--Lehel theorem can be used to show that $g_{3,2}(N) \leq N^{2/3}$ whenever $N$ is a perfect cube, and hence $g_{3,2}(N) \leq (1+o(1))N^{2/3}$ in general. In fact, the example witnessing this is a $3$-edge-colored \emph{transitive} tournament; the case of transitive tournaments plays an important role (and was Loh's main object of study), so it is natural to further define $f_{q,r}(N)$ identically to $g_{q,r}(N)$, except restricting our attention only to $q$-edge-colorings of the $N$-vertex transitive tournament. We thus have $g_{q,r}(N) \leq f_{q,r}(N)$ for all $q,r,N$, and the bounds discussed above show that
\begin{equation}\label{eq:32 bounds}
    \sqrt N \leq g_{3,2}(N) \leq f_{3,2}(N) \leq (1+o(1))N^{2/3}.
\end{equation}
It seems to be surprisingly difficult to meaningfully improve either of these bounds, and Loh was only able to show a very minor asymptotic improvement on the lower bound for $f_{3,2}(N)$, proving that $f_{3,2}(N) \geq \sqrt N \cdot 2^{\Omega(\log^* N)}$, where $\log^*$ is the iterated logarithm function. This is proved via a reduction to the famous triangle removal lemma of Ruzsa and Szemer\'edi \cite{MR519318}, and the precise bound stated above follows from the strongest known quantitative version of the triangle removal lemma, due to the first author \cite{MR2811609}. 

Loh's work inspired a number of further authors \cite{MR4298506,1608.04153,MR4975150,MR3684779} to study this kind of question. Notably, Gowers and Long \cite{MR4298506} managed to significantly improve Loh's lower bound for $f_{3,2}(N)$ by a polynomial amount, showing that $f_{3,2}(N) \geq N^{\frac 12 + \delta}$ for an extremely small, but absolute, constant $\delta>0$. Their proof is based on an induction argument, and the base case of the induction is established by applying the bound $f_{3,2}(N) \geq \sqrt N \cdot 2^{\Omega(\log^* N)}$ discussed above; since this improvement only beats the trivial bound for an enormous $N$, the constant $\delta$ they obtain is tiny. To the best of our knowledge, there is no known asymptotic improvement on the lower bound $g_{3,2}(N) \geq \sqrt N$. 

When it comes to upper bounds on $f_{3,2}(N)$, nothing better than $(1+o(1))N^{2/3}$ is known, and Gowers and Long \cite{MR4298506} conjectured that $f_{3,2}(N) \geq N^{2/3}$. However, rather surprisingly, better upper bounds on $g_{3,2}(N)$ are available, and have been known (in a disguised form) since the 1980s. Indeed, Hamaker and Stein \cite{MR754867} found a construction for the closely related \emph{tripod packing} problem (discussed further below) which, when converted back to our setting, yields a $3$-edge-colored $N$-vertex tournament whose longest color-avoiding directed path has length $N^{0.663}$. 
The exponent in the upper bound has been steadily improved (e.g.\ \cite{MR2326159,intelligencer,MR1311249}). Unaware of these prior developments, Gowers and Long \cite{MR4298506} developed a continuous relaxation of this problem, allowing them to obtain current record of $g_{3,2}(n) \leq N^{0.649}$; they even speculate, based on some computational evidence of Wagner, that their construction may be optimal. We stress that all of these constructions use highly non-transitive tournaments, and that the best known upper bound in the transitive case remains $f_{3,2}(N) \leq (1+o(1))N^{2/3}$, which is conjecturally optimal; thus, in contrast to the Chv\'atal--Gy\'arf\'as--Lehel theorem, it seems that the transitive case is not extremal for this color-avoiding problem. This is a genuine difference from the monochromatic case, and may serve to explain some of the difficulty of this problem.

In this paper, we study this kind of color-avoiding question in $q$-colored tournaments, where $q \geq 4$. Extending \eqref{eq:32 bounds}, the simple known bounds on this problem are
\[
    \sqrt N \leq g_{q,q-1}(N) \leq f_{q,q-1}(N) \leq (1+o(1)) N^{1-1/q},
\]
where the upper bound follows from the same product construction and the lower bound again follows by arbitrarily merging the colors into two classes and applying the Chv\'atal--Gy\'arf\'as--Lehel theorem. Note that the gap between the upper and lower bounds only grows as $q$ increases; moreover, the lower bound appears to be more and more wasteful as $q$ grows. It is thus natural to expect that the true exponent on $N$ should increase with $q$, and perhaps even converge to $1$ as $q \to \infty$; this is precisely the statement of our main theorem.

\begin{theorem}\label{thm:main}
    We have that $g_{q,q-1}(N) \geq c_q N^{1-O(1/\sqrt{\log q})}$, where $c_q >0$ is a constant depending only on $q$.
    That is, in any $q$-edge coloring of any $N$-vertex tournament, there exists a color-avoiding directed path of length at least $c_q N^{1-O(1/\sqrt{\log q})}$.
\end{theorem}

In particular, we obtain the same lower bound when restricting to transitive tournaments; however, as we discuss below, alternative techniques give even stronger lower bounds on $f_{q,q-1}(N)$. Nevertheless, rather interestingly, our proof approach demonstrates that any extremal construction for $g_{q,q-1}(N)$ for $q \geq 4$ must be approximately transitive.
To make this precise, let us recall that an $N$-vertex tournament is \emph{$\delta$-close to transitive} if it can be made transitive by reversing the orientation of at most $\delta N^2$ edges, and is \emph{$\delta$-far from transitive} otherwise.
\begin{theorem}\label{labelfarfromtransitive}
    If an $N$-vertex tournament is $\delta$-far from transitive, then any $q$-coloring of its edges contains a directed path of length $c\delta^2 N/q^3$ which is colored by at most three colors, where $c>0$ is an absolute constant. Furthermore, the coloring contains the square of a path such that the direction of the edges between consecutive vertices is forward, the direction of edges between vertices of distance two is backwards, and the coloring of the edges is periodic with period $3$. 
\end{theorem}
\begin{center}
    \begin{tikzpicture}
        \foreach \x in {1,...,15} \node[vert] (\x) at (\x,0) {};
        \foreach \x in {1,4,...,15} {
        \pgfmathtruncatemacro\y{\x+1}
        \pgfmathtruncatemacro\z{\x-2}
        \draw[ar, red] (\x) -- (\y);
        \ifthenelse{\z>0}{\draw[ar, red, arpos=.5] (\x) to[out=150, in=30] (\z);}{}
        }
        
        \foreach \x in {2,5,...,15} {
        \pgfmathtruncatemacro\y{\x+1}
        \pgfmathtruncatemacro\z{\x-2}
        \draw[ar, blue] (\x) -- (\y);
        \ifthenelse{\z>0}{\draw[ar, blue, arpos=.5] (\x) to[out=150, in=30] (\z);}{}
        }
        
        \foreach \x in {3,6,...,15} {
        \pgfmathtruncatemacro\y{\x+1}
        \pgfmathtruncatemacro\z{\x-2}
        \ifthenelse{\y<16}{\draw[ar, green!50!black] (\x) -- (\y);}{}
        \ifthenelse{\z>0}{\draw[ar, green!50!black, arpos=.5] (\x) to[out=150, in=30] (\z);}{}
        }
    \end{tikzpicture}
\end{center}

In particular, if $q \geq 4$ is fixed and $\delta = \Omega(1)$, then we find a color-avoiding directed path of length $\Omega(N)$, much stronger than the result in \cref{thm:main}. We will use this in our proof of \cref{thm:main}, as it allows us to focus on the nearly transitive case. We stress that there is a genuine difference between the $q=3$ and $q\geq 4$ cases: the construction of Gowers--Long is $\Omega(1)$-far from transitive, yet does not have color-avoiding directed paths of linear length. Moreover, their construction has much shorter color-avoiding paths than the best known, and conjecturally extremal, transitive construction. We also stress that, in general, we do not expect transitive tournaments to be the ones minimizing the length of color-avoiding paths.

The proof of \cref{labelfarfromtransitive} is extremely short and simple, and relies on a result of the first two authors \cite{FoSu} on the structure of tournaments that are far from transitive. They proved a directed version of the triangle removal lemma \cite{MR519318} with very good quantitative bounds, and it is this good quantitative behavior that means that the $\delta$-dependence in \cref{labelfarfromtransitive} is polynomial.

\subsection{Applications: vectors sequences and pod packings}
All prior works studying color-avoiding directed paths have used a close connection to certain geometric questions, whose origin can be traced to Seidenberg's classical proof \cite{MR106189} of the Erd\H os--Szekeres lemma. These geometric questions themselves have a long and storied history, and
have been extensively studied for over 50 years, with different researchers arriving at them from a number of different perspectives,
including extremal problems in convex geometry \cite{MR2065250} and hypergraphs \cite{MR4298506}, $k$-majority tournaments \cite{1505.07312}, incidence geometry \cite{1608.04153}, the Erd\H os--Hajnal conjecture \cite{MR3684779}, and packing questions motivated by group theory \cite{MR219435} and error-correcting codes \cite{MR754867}. For a detailed discussion of these disparate instantiations, we refer to \cite[Section 4.3]{MR3904824}.

To describe this connection, we make the following definition.
Given integers $q \geq r \geq 1$, we define the relation $<_r$ on $\R^q$ by saying that $x<_r y$ if there are at least $r$ coordinates $i \in [q]$ for which $x_i<y_i$. Despite the notation, $<_r$ is not a partial order whenever $r<q$, as it is not transitive. 

In order to explain how this relates to color-avoiding directed paths, let us focus for the moment on the transitive setting: say we are given a transitive tournament with vertex set $[N]$ and all edges oriented forwards, and suppose each of its edges has been assigned a color from $[q]$. For every vertex $a \in [N]$, we define a vector $x_a \in \N^q$ as follows: the $i$th coordinate of $x_a$ records the length of the longest directed path ending at $a$ whose edges avoid the color $i$. The property enjoyed by these vectors is that they form a \emph{$(q-1)$-increasing sequence}. That is, for every $a<b$, we have that $x_a <_{q-1} x_b$. Indeed, if the color of the edge $ab$ is $i$, then for every $i \neq j \in [q]$, we may extend the longest $j$-avoiding path ending at $a$ to a strictly longer $j$-avoiding path ending at $b$, implying that $(x_b)_j > (x_a)_j$. Moreover, if $n$ is the length of the longest color-avoiding directed path in the given transitive tournament, then all the vectors $x_a$ lie in the box $[n]^q$. 

We are thus led to the natural extremal function $F_{q,q-1}(n)$, defined as the length of the longest $(q-1)$-increasing sequence of vectors in $[n]^q$. The argument presented above shows if $f_{q,q-1}(N) \leq n$, then $F_{q,q-1}(n) \geq N$, as we may convert any transitive tournament with no long color-avoiding paths into a long $(q-1)$-increasing sequence. Somewhat surprisingly, this connection can be reversed: given a $(q-1)$-increasing sequence of $N$ vectors in $[n]^q$, one can $q$-color the edges of an $N$-vertex transitive tournament such that every color-avoiding path has length at most $n$. Indeed, for $a<b$ one colors the edge $ab$ by the unique index $i$ (if it exists) such that $(x_a)_i>(x_b)_i$; if there is no such index, this edge can be colored arbitrarily. In this coloring, any directed path avoiding the color $i$ must have strictly increasing $i$th coordinates in the corresponding vectors, implying that its length is at most $n$ since the vectors are contained in $[n]^q$. Combining these two constructions, we see that the extremal functions $F_{q,q-1}$ and $f_{q,q-1}$ are inverses of one another, in the sense that
\begin{equation}\label{eq:Ff equivalence}
    f_{q,q-1}(N) \leq n\qquad \text{if and only if}\qquad F_{q,q-1}(n)\geq N.
\end{equation}
For a formal proof of this fact, see e.g.\ \cite[Lemma 2.1]{1505.07312}.

Of course, our study in this paper is not restricted to transitive tournaments, so it is natural to generalize the notion of $(q-1)$-increasing sequences as follows. We say that a set $\{x_1,\dots,x_N\} \subseteq [n]^q$ is \emph{$(q-1)$-comparable} if, for all $a<b$, we have $x_a<_{q-1} x_b$ or $x_b <_{q-1} x_a$, and we let $G_{q,q-1}(N)$ denote the maximum size of a $(q-1)$-comparable set of vectors in $[n]^q$. Even more generally, for $q \geq r \geq 1$, we define $F_{q,r}(n)$ and $G_{q,r}(n)$ analogously, as the maximum length (resp.\ maximum size) of an $r$-increasing sequence (resp.\ $r$-comparable set) in $[n]^q$.
Given an $r$-comparable set, we naturally obtain an $N$-vertex tournament\footnote{Strictly speaking, one only obtains a well-defined tournament if $r>q/2$, as this condition guarantees that if $x <_r y$, then $y \not<_r x$. If $r \leq q/2$, then we can, for example, allow antiparallel edges in case $x<_r y$ and $y<_r x$.}, recording for each pair $(a,b)$ whether $x_a <_r x_b$ or $x_b <_r x_a$. Thus, the $r$-comparable set is in fact an $r$-increasing sequence if and only if this tournament is transitive. Additionally, the fact that the $r$-comparable condition is weaker immediately implies that $F_{q,r}(n) \leq G_{q,r}(n)$ for all $q,r,n$.

Before proceeding, we make two warnings about the difficulty of extending \eqref{eq:Ff equivalence}. First, it is very natural to assume that this equivalence carries over verbatim to an equivalence between $F_{q,r}$ and $f_{q,r}$. However, this is not the case. 
Instead, the same argument above shows that $F_{q,r}(n)$ is equal to the largest $N$ for which the following holds. One can assign to every edge of an $N$-vertex transitive tournament a subset of $r$ colors from a palette of $q$ colors, and no color appears in the color set of every edge of a directed path $P$ of length $n+1$. This is one less than the \textit{set-coloring ordered Ramsey number} of a monotone path with $n+1$ vertices. Set-coloring Ramsey numbers have been extensively studied recently for cliques (see \cite{MR4752601,MR4751635,MR4704267}) and have found interesting connections to coding theory, and have also been studied for certain other classes of graphs (see e.g.\ \cite{MR4563218} and the references therein). Ordered Ramsey numbers of graphs have also been extensively studied recently (see \cite{MR3575208,MR4057168} and the recent survey \cite{2502.02155}). 

Our second warning concerns the extension of \eqref{eq:Ff equivalence} to the non-transitive setting. Again, it would be natural to expect that $G_{q,q-1}$ and $g_{q,q-1}$ are inverses of each other, but this is not the case. The thing that goes wrong is the first direction of the equivalence: it may be that the edge $a\to b$ is colored with color $i$, but we can nevertheless not extend a $j$-avoiding path ending at $a$ to one ending at $b$, as such a path may have already visited $b$. Note that we did not encounter this issue in the transitive setting, as such a situation can only arise in the presence of directed cycles. Although this issue may seem like a mere technicality, it turns out to be fundamental; it is also the main difficulty in the proof of the Gallai--Hasse--Roy--Vitaver theorem, and must also be carefully dealt with in our proof of \cref{thm:main}. Nevertheless, the other direction of the equivalence does go through in the non-transitive setting: any $(q-1)$-comparable set of vectors in $[n]^q$ yields a $q$-coloring of their corresponding tournament in which the longest directed color-avoiding path has at length at most $n$. That is, we have the one way implication
\begin{equation}\label{eq:Gg equivalence}
    \text{if }G_{q,q-1}(n) \geq N, \qquad \text{then }g_{q,q-1}(N) \leq n.
\end{equation}
Thus, for example, the constructions of Hamaker--Stein \cite{MR754867} and Gowers--Long \cite{MR4298506} actually prove lower bounds on $G_{3,2}(n)$, from which one deduces the upper bounds on $g_{3,2}(N)$ which we discussed above.

In fact, the work of Hamaker--Stein, as well as many  subsequent papers on this topic, study yet another geometric question which we now introduce. 
A {\it tripod} of order $n$ consists of a corner and the three adjacent edges of an integer $n \times n \times n$ cube. The corner of the tripod is called its \textit{apex}. We say that a collection of tripods is \textit{aligned} if they are translates of one other, and that they \emph{pack} if they are pairwise disjoint. The maximum number of aligned tripods which pack and whose apex is in $[n]^3$ is precisely $G_{3,2}(n)$, as two tripods with apices in $[n]^3$ are disjoint if and only if their apices are $2$-comparable. 
Most of the early paper studying $G_{3,2}(n)$ were motivated
from the question of packing tripods, 
and it is not hard to show that the two questions are equivalent: bounding the maximum number of disjoint tripods of order $n$ whose apex lies in $[n]^3$ is equivalent to bounding the maximum density of $3$-space that can be packed by order-$n$ tripods. For a formal statement (and in particular for the formal definition of the density of a packing), we refer to \cite[Chapter 4]{MR1311249} or \cite{MR754866}. 

More generally, a {\it $(q,r)$-pod of order $n$} consists of a corner and the incident $(r-1)$-dimensional faces of an integer $q$-dimensional cube with side length $n$. We may just refer to it as a $(q,r)$-pod if the order $n$ is implicit. As above, the corner of the $(q,r)$-pod is called its \textit{apex}, and a collection of $(q,r)$-pods is \textit{aligned} if they are translates of each other, and the collection \textit{packs} if they are pairwise disjoint. The maximum number of aligned $(q,r)$-pods that pack whose apex is in $[n]^q$ is precisely $G_{q,r}(n)$ as two $(q,r)$-pods with apices in $[n]^q$ are disjoint if and only if their apices are $r$-comparable. Indeed, if $(x_1,\ldots,x_q),(y_1,\ldots,y_q) \in [n]^q$ are not $r$-comparable, then $(z_1,\ldots,z_q) \in [n]^q$ given by $z_i=\max(x_i,y_i)$ for each coordinate $i$ is in both the $(q,r)$-pod with apex $(x_1,\ldots,x_q)$ and the $(q,r)$-pod with apex $(y_1,\ldots,y_q)$. Conversely, if $(x_1,\ldots,x_q),(y_1,\ldots,y_q)$ are $r$-comparable, then  without loss of generality, $y_i>x_i$ for $i=1,\ldots,r$, in which case every element of the $(q,r)$-pod with apex $(y_1,\ldots,y_q)$ has value at least $y_i$ in coordinate $i$ for each $i \in [r]$, but no element of the $(q,r)$-pod with apex $(x_1,\ldots,x_q)$ has this property, so these $(q,r)$-pods are disjoint. 

\begin{figure}[ht]
    \newcommand{\cube}[3]{%
  \begin{scope}[shift={(#1,#2,#3)}]
    \filldraw[fill=white!90!gray,draw=black]
      (0,0,1) -- (1,0,1) -- (1,1,1) -- (0,1,1) -- cycle;
    \filldraw[fill=white!70!gray,draw=black]
      (0,0,0) -- (0,1,0) -- (0,1,1) -- (0,0,1) -- cycle;
    \filldraw[fill=white!50!gray,draw=black]
      (0,0,0) -- (1,0,0) -- (1,0,1) -- (0,0,1) -- cycle;
  \end{scope}%
}
\def\tripod{  0/0/0, 0/0/1, 0/0/2, -1/0/0, -2/0/0, 0/-1/0, 0/-2/0}
\begin{subfigure}[t]{.23\textwidth}
    \centering
    \begin{tikzpicture}[
  scale=.55,
  x={(0.9cm,0.4cm)},
  y={(-0.9cm,0.4cm)},
  z={(0cm,1cm)},
  line join=round,
  line cap=round
]
\foreach \x/\y/\z in {3/0/0, 2/0/0, 1/0/0,  0/0/0, 0/0/1, 0/0/2, 0/0/3, 0/-1/0, 0/-2/0, 0/-3/0} {\cube{\x}{\y}{\z}}
\end{tikzpicture}
\caption{A tripod of order 4}
\end{subfigure}
\begin{subfigure}[t]{.72\textwidth}
    \centering

\begin{tikzpicture}[
  scale=.55,
  x={(0.9cm,0.4cm)},
  y={(-0.9cm,0.4cm)},
  z={(0cm,1cm)},
  line join=round,
  line cap=round
]
\foreach \sh in {1,2} {
\begin{scope}[xshift=190*\sh]
    \foreach \x/\y/\z in \tripod {\cube{\x}{\y}{\z}}
    \scoped[x={(1cm,0)}, y={(0,1cm)}] \draw (-3,-2) rectangle (3,4);
    \scoped[x={(1cm,0)}, y={(0,1cm)}] \node at (0,4.5) {last coordinate = \sh};
\end{scope}
\foreach \x/\y/\z in {0/0/0, 0/0/1, 0/0/2, -1/0/0, -1/0/1, -1/0/2, -2/0/0, -2/0/1, -2/0/2, 0/-1/0, 0/-1/1, 0/-1/2, 0/-2/0, 0/-2/1, 0/-2/2, -1/-1/0, -1/-2/0, -2/-1/0, -2/-2/0} {\cube{\x}{\y}{\z}}
}
    \scoped[x={(1cm,0)}, y={(0,1cm)}] \draw (-3,-2) rectangle (3,4);
    \scoped[x={(1cm,0)}, y={(0,1cm)}] \node at (0,4.5) {last coordinate = 0};
\end{tikzpicture}
\caption{A $(4,3)$-pod of order $3$; each square represents a $3$-dimensional slice of $\mathbb R^4$}
\end{subfigure}
\end{figure}

Gowers and Long \cite{MR4298506} proved a large number of results on the extremal functions $F_{q,r}(n)$ and $G_{q,r}(n)$. For example, we already mentioned their improved bound $f_{3,2}(N) \geq N^{\frac 12 + \delta}$; they actually proved it in the equivalent form $F_{3,2}(n) \leq n^{2-\delta}$. It is easy to deduce from this that $F_{q,q-1}(n) \leq n^{2-\delta}$ for all $n$, again obtaining a power savings over the trivial bound of $n^2$. However, for the more general question of $G_{q,r}(n)$, their results were much more limited: the trivial upper bound here\footnote{Indeed, by deleting the least $r-1$ coordinates from every vector in an $r$-comparable set, we obtain a $1$-comparable set in $[n]^{q-r+1}$, whose size is trivially at most $n^{q-r+1}$.} is $G_{q,r}(n) \leq n^{q-r+1}$, and they noted that an application of the hypergraph removal lemma, generalizing the use of the triangle removal lemma above\footnote{In the setting of vectors, the fact that the removal lemma is applicable was first noted by Tiskin \cite{MR2326159}, several years before Loh \cite{1505.07312} used the same idea in the study of color-avoiding paths.}, shows that $G_{q,r}(n) = o(n^{q-r+1})$. However,
the savings is extraordinarily small (the savings is roughly the reciprocal of the $(q-r+1)$th level of the Ackermann hierarchy, evaluated at $n$). As such, Gowers and Long highlighted the following question, saying it is the ``most annoying'' question they were unable to answer. 
\begin{question}[{\cite[Question 6.3]{MR4298506}}]\label{qu:annoying}
Is there any pair $(q,r)$ for which one can prove a power savings $G_{q,r}(n) \leq n^{q-r+1-\delta}$, for some absolute constant $\delta>0$?
\end{question}
Thanks to \eqref{eq:Gg equivalence}, our lower bound on $g_{q,q-1}(N)$ in \cref{thm:main} yields a corresponding upper bound on $G_{q,q-1}(n)$, and thus also on the packing density of $(q,q-1)$-pods. As such, we are able to resolve \cref{qu:annoying}, as stated in the following corollary.
\begin{corollary}\label{cor:applications}
    We have
    \[
        G_{q,q-1}(n) \leq C_q n^{1+O(1/\sqrt{\log q})},
    \]
    where $C_q>0$ is an absolute constant. 

    Consequently, for $n$ sufficiently large as a function of $q$, the density of $q$-space that can be packed by $(q,q-1)$-pods of order $n$ is at most $n^{O(1/\sqrt{\log q})-1}$. 
\end{corollary}
In the second statement, we used the fact that there are only $2^q$ corners of the $q$-cube, so by losing a factor $2^q$, we can assume that all the $(q,q-1)$-pods we pack are aligned.

However, if one is only interested in these geometric corollaries, there is an alternative approach that yields much stronger bounds. The first step of the approach is the following surprisingly simple observation, which shows that when $r>\frac{2q}3$, there is no difference between the $r$-increasing and $r$-comparable questions.

\begin{proposition}\label{prop:F=G}
    Let $q\geq r$ be integers with $r>\frac{2q}3$. Then $F_{q,r}(n)=G_{q,r}(n)$ for all $n$. 
\end{proposition}

\cref{prop:F=G} is quite surprising for a few reasons. First, it is false for $(q,r)=(3,2)$: the old construction of Hamaker and Stein shows $G_{3,2}(7)\geq 19$, whereas one can check by computer that $F_{3,2}(7)= 17$. More generally, Gowers and Long proved that $G_{3,2}(n) \geq n^{1.54}$, and they conjectured that $F_{3,2}(n) \leq n^{3/2}$. These facts also suggest that the requirement $r>\frac{2q}3$ in \cref{prop:F=G} is best possible. 

The other reason why \cref{prop:F=G} is surprising to us is that we do not expect it to be true for tournaments, i.e.\ we do not expect $f_{q,q-1}(N)=g_{q,q-1}(N)$; although \cref{labelfarfromtransitive} shows that the extremal examples for $g_{q,q-1}$ must be close to transitive, we do not expect them to be exactly transitive.
Despite all of this, the proof of \cref{prop:F=G} is extremely simple.
\begin{proof}[Proof of \cref{prop:F=G}]
    Let $S \subseteq [n]^q$ be an $r$-comparable set with $\ab S = G_{q,r}(n)$. If the tournament defined by $S$ is transitive, then $\ab S \leq F_{q,r}(n)$, and we are done. If not, then this tournament has a cyclic triangle, that is, three vectors $x,y,z \in [n]^q$ such that $x <_r y <_r z <_r x$. Let
    \[
        A = \{i \in [q]: x_i < y_i \} ,\qquad B = \{i \in [q]: y_i < z_i \} ,\qquad C = \{i \in [q]: z_i < x_i \} .
    \]
    Then $\ab A, \ab B, \ab C \geq r > \frac{2q}{3}$, hence $A \cap B \cap C \neq \varnothing$. But if $i \in A \cap B \cap C$, we have $x_i < y_i < z_i < x_i$, which is a contradiction.
\end{proof}
\cref{prop:F=G} is already enough to answer \cref{qu:annoying}, when combined with the result $F_{q,q-1}(n) \leq n^{2-\delta}$ of Gowers and Long. However, as pointed out to us by Cosmin Pohoata (personal communication), much stronger results were obtained by Sawin and Tao \cite{206224}, who proved the following theorem.
\begin{proposition}[Sawin--Tao]\label{prop:sawin-tao}
    Let $\gamma \in (\frac 12, 1)$, and let $(X_1,\dots,X_N)$ be a sequence of (dependent) random variables valued in $[n]$ with the property that $\mathbb P(X_b>X_a) \geq \gamma$ for all $1 \leq a<b\leq N$. Then
    \[
    N \leq n^{1/(2\gamma-1)}.
    \]
\end{proposition}
As noted by Pohoata, this immediately implies that 
for all $q \geq r \geq 1$ with $r>\frac q2$, we have
\begin{equation}\label{eq:sawin-tao bound}
    F_{q,r}(n) \leq n^{\frac{q}{2r-q}}.
\end{equation}
Indeed, suppose that $x_1,\dots,x_N \in [n]^q$ is an $r$-increasing sequence of vectors with $N=F_{q,r}(n)$. Let $Y$ be a uniformly random element of $[q]$, and let $X_a$ equal the $Y$\textsuperscript{th} coordinate of $x_a$, for all $a$. Then $(X_1,\dots,X_n)$ is a sequence of (dependent) random variables valued in $[n]$, and the $r$-increasing condition implies that $\mathbb P(X_b> X_a) \geq r/q$ for all $a<b$, since $x_b$ is strictly larger than $x_a$ in at least $r$ of its $q$ coordinates. We may thus apply \cref{prop:sawin-tao} to conclude that $F_{q,r}(n)=N \leq n^{1/(2(r/q)-1)}$, yielding \eqref{eq:sawin-tao bound}.
In particular, applying this bound with $r=q-1$, we find that
\[
    F_{q,q-1}(n) \leq n^{1+ \frac{2}{q-2}}
\]
for all $n$. When $q \geq 4$, we can combine this with \cref{prop:F=G} to conclude that $G_{q,q-1}(n) \leq n^{1+\frac{2}{q-2}}$ as well, yielding a much stronger dependence on $q$ than given in \cref{cor:applications}.

The proof of \cref{prop:sawin-tao} uses Fourier-analytic methods, and as such seems restricted to the setting of $r$-increasing vector sequences. Thanks to the equivalence \eqref{eq:Ff equivalence}, we obtain from it a strong bound on $f_{q,q-1}(N)$ as well. However, we do not see how to extend these analytic techniques to handle more general tournaments, and in particular we do not see how to deduce any bound on $g_{q,q-1}(N)$ using these techniques. Thus, while \cref{thm:main} yields weaker bounds in these applications, it appears to be a genuinely more general result.

The rest of this paper is organized as follows. We present a high-level outline of the proof of \cref{thm:main} in \cref{sec:overview}. In \cref{sec:simple}, we warm up to the topic by presenting some simple proofs, including the short proof of \cref{labelfarfromtransitive} and some properties of the product construction alluded to previously. We present the proof of \cref{thm:main} in \cref{sec:main proof}, and end in \cref{sec:conclusion} with some concluding remarks. 

In this paper, all logarithms are to base $2$. We systematically omit floor and ceiling signs whenever they are not crucial.

\section{Proof overview}\label{sec:overview}
In this section, we give a sketch of the proof of \cref{thm:main}. For the majority of our discussion, we focus on the case of \emph{transitive} tournaments; most of the ideas used in the proof of the full theorem are already present in this setting, with fewer technical complications. Although the Fourier-analytic techniques of Sawin and Tao yield stronger bounds in the transitive setting, they do not seem to extend to the general case, in contrast to the combinatorial techniques we use. At the end of the sketch, we describe which additional ingredients are needed for the general case.

As our tournament is transitive, we can simply view it as an ordered complete graph on the vertex set $[N]$; additionally, a path in this complete graph is directed if and only if it is monotone, i.e.\ its vertices are strictly increasing as integers in $[N]$. So our task is to prove that any $q$-edge-colored complete graph on $[N]$ contains a color-avoiding monotone path of length at least $c_q N^{1-C/\sqrt{\log q}}$, and we prove this by induction on $N$. For this proof sketch, we will not worry too much about the precise quantitative error term, and only aim to prove a lower bound of the form $N^{1-o(1)}$, for some function $o(1)$ that tends to $0$ as $q \to \infty$. For a $q$-edge-colored transitive tournament $T$ and a color $i \in [q]$, let $\ell_i(T)$ denote the length of the longest directed path in $T$ avoiding color $i$. Thus, our goal is to prove by induction that $\ell_i(T) \geq N^{1-o(1)}$ for some $i$. We will actually strengthen the induction hypothesis, and prove that $\Pi(T) \geq N^{q-o(q)}$, where we define $\Pi(T)=\prod_{i=1}^q \ell_i(T)$. The statement $\Pi(T) \geq N^{q-o(q)}$ immediately implies that $\ell_i(T) \geq N^{1-o(1)}$ for some $i \in [q]$, hence this really is a stronger statement.

Our strategy is as follows. We divide $[N]$ into two subintervals $T_1 = [1,N/2]$ and $T_2= (N/2,N]$, and apply the inductive hypothesis in each interval. Our dream scenario is to take a longest $i$-avoiding path $P_1$ in the first half and a longest $i$-avoiding path $P_2$ in the second half, and to glue them together to obtain an $i$-avoiding path of double the length. Of course, for this to work, we need the edge joining the end of $P_1$ and the start of $P_2$ to \emph{not} be colored with color $i$; if we can ensure this, then their concatenation $P_1 \cup P_2$ truly is an $i$-avoiding path of length $\ab{P_1}+\ab{P_2}$. Thus, if we can do this for color $i$, we learn that
\[
    \ell_i(T) \geq \ell_i(T_1)+\ell_i(T_2) \geq 2 \sqrt{\ell_i(T_1)\ell_i(T_2)},
\]
where the final step is the inequality of arithmetic and geometric means. If, in turn, we can do this for \emph{all} the colors, we conclude that
\[
    \Pi(T) =\prod_{i=1}^q \ell_i(T) \geq \prod_{i=1}^q 2\sqrt{\ell_i(T_1)\ell_i(T_2)} \geq 2^q \cdot \Pi_q(N/2),
\]
where $\Pi_q(N/2)$ is our inductive lower bound on $\Pi(T')$ for every $q$-edge-colored transitive tournament $T'$ on $N/2$ vertices. Inductively, we know that $\Pi_q(N/2)\geq (N/2)^{q-o(q)}$, hence we learn that $\Pi(T) \geq 2^q (N/2)^{q-o(q)} \geq N^{q-o(q)}$. That is, in the dream scenario, we can prove the desired result.

Of course, there is no reason for the dream scenario to be remotely close to true---we shouldn't even expect to be able to glue together a longest $i$-avoiding path in each half even for \emph{a single} color $i$, let alone for all of them. The proof now consists of showing that either we can glue together \emph{almost} longest $i$-avoiding paths for \emph{almost} all colors $i$, or else that we win for another reason; the losses inherent in these ``almosts'' are what contribute to the error term in the exponent. 

We begin by explaining the first ``almost''. Because we are allowed to tolerate errors, it is not necessary for $P_1$ and $P_2$ above to be truly the longest $i$-avoiding paths in $T_1,T_2$, respectively. Instead, it suffices that their lengths are nearly as long as possible, so that $\ell_i(T) \geq (2-o(1))\sqrt{\ell_i(T_1)\ell_i(T_2)}$. This gives us a lot more flexibility, since it now suffices to just find one such pair $P_1,P_2$, of nearly maximal length, such that the edge joining the end of $P_1$ and the start of $P_2$ is not colored with color $i$. To accomplish this, we will consider a large number of nearly maximal $i$-avoiding paths in each half, and attempt to glue together every pair.

Concretely, set $s = \varepsilon N/q$, where $\varepsilon>0$ is a small parameter depending only on $q$ (and tending to $0$ as $q \to \infty$). We wish to pick $s$ different $i$-avoiding paths in each half; moreover, since our goal is to try gluing every pair of paths from each half, we want these $s$ paths to have different endpoints. To accomplish this, let us define $\ell_i(^\to v)$ to be the length of the longest $i$-avoiding directed path ending at $v$, and similarly $\ell_i(w^\to)$ to be the length of the longest $i$-avoiding directed path starting at $w$. We now define $X_i$ to be the set of the $s$ vertices in the first half with the highest value of $\ell_i(^\to \bullet)$. Similarly, we define $Y_i$ to be the set of the $s$ vertices in the second half with the highest value of $\ell_i(\bullet^\to)$. The paths $P_1$ that we consider will then be the longest $i$-avoiding paths ending at a vertex of $X_i$, and similarly the $P_2$ we consider are the longest $i$-avoiding paths starting at a vertex of $Y_i$. Heuristically, since $s \ll N$ and $X_i$ consists of the $s$ ``best'' endpoints for $i$-avoiding paths in $T_1$, every vertex of $X_i$ is the endpoint of such a path of length almost $\ell_i(T_1)$.

We now call a color $i$ \emph{compressed} if the majority of the vertices of $X_i$ and the majority of the vertices of $Y_i$ are close to the midpoint $N/2$, i.e.\ in the interval $[\frac N2 - 4s, \frac N2 + 4s]$. The claim now is that we can glue together $i$-avoiding paths in $T_1$ and $T_2$ for nearly all compressed colors. Indeed, if $i$ is a compressed color and we cannot perform such a gluing, that means that all edges from $X_i$ to $Y_i$ are of color $i$, and in particular we have found many edges of color $i$ within the interval $[\frac N2 - 4s, \frac N2 + 4s]$. As there are only $O(s^2)$ edges within this interval, we must be able to glue for the vast majority of compressed colors: each color in which we cannot contributes many edges in this interval, and there are not so many such edges.

Thus, if nearly all colors are compressed, we are ``almost'' in the dream scenario and we can prove the claim inductively by absorbing all the losses in the $N^{-o(q)}$ factor. It remains to understand what happens when we have a large number, say $2p$, of non-compressed colors; by symmetry, we may assume that at least $p$ of them are \emph{left-diffuse}, i.e.\ that the majority of $X_i$ for these colors lies in the interval $[1,\frac N2-4s]$. Let $U_i =X_i \cap [1, \frac N2-4s]$, for each of these $p$ left-diffuse colors $i$.

The key claim now is the following. Suppose that $v,w$ are vertices, where $v$ precedes $w$ in the ordering of $[N/2]$, and assume that $v \in X_i, w \notin X_i$. Then the edge $vw$ is necessarily colored with color $i$. Indeed, if it were not, then we can extend any $i$-avoiding path ending at $v$ to a longer one ending at $w$, but then the fact that $v \in X_i,w\notin X_i$ contradicts the definition of $X_i$. This simple observation has two immediate corollaries. First, if $i,j$ are distinct left-diffuse colors, then $U_i$ and $U_j$ are disjoint. Indeed, since there are few vertices of $X_i \cup X_j$ in $[\frac N2-4s,\frac N2]$, we can find $w\notin X_i \cup X_j$ in this interval. If $v \in U_i \cap U_j$, then the observation above implies that $vw$ must receive both color $i$ and color $j$, which is impossible, hence $U_i \cap U_j=\varnothing$. The second corollary is that the color of every edge between $U_i$ and $U_j$ must be left-diffuse; indeed, by the key claim, every such edge must in fact be either of color $i$ or of color $j$.

This immediately has the following remarkable consequence. For every color $k$ which is \emph{not} one of our $p$ left-diffuse colors, we have
\[
    \ell_k(T) \geq \sum_{i\text{ left-diffuse}} \ell_k(T[U_i]) \geq p\left(\prod_{i\text{ left-diffuse}} \ell_k(T[U_i])\right)^{1/p}.
\]
Indeed, we get the desired $k$-avoiding path by using the vertices from the longest $k$-avoiding path in each $U_i$. This path is $k$-avoiding as every edge of the path is either between vertices from the same $U_i$ or must be one of the $p$ left-diffuse colors. The last inequality is the inequality of arithmetic and geometric means. In other words, we find that for $q-p$ choices of color $k$, we essentially gain a factor of $p$ in the length of the longest $k$-avoiding path. This is an enormous gain, much larger than the factor of $2$ we were winning in the dream scenario; moreover, by choosing $p = o(q)$, we can obtain such a gain for nearly all colors. The loss comes from the fact that we now need to apply the inductive hypothesis to the subtournament induced on $\bigcup_{i \text{ left-diffuse}} U_i$, which is quite small; however, by picking $p$ appropriately, we can ensure that the amount we win by multiplying the lengths by $p$ dominates this loss.

Apart from the computations, the discussion above is an essentially complete sketch of the proof of \cref{thm:main} in the case that $T$ is transitive. We now discuss the difficulties encountered when passing to the general setting, and the main additional ideas needed to overcome them. 

First, by applying \cref{labelfarfromtransitive}, we may assume that $T$ is $\delta$-close to transitive, for some small constant $\delta=\delta(q)>0$. Indeed, if this is not the case, then \cref{labelfarfromtransitive} shows that $T$ contains a path of linear length which receives at most $3$ colors, a far stronger statement than what we aim to prove. We now fix an ordering of $V(T)$ such that all but at most $\delta N^2$ edges go forward. In fact, by deleting a small number of vertices, we can even pass to a ``minimum degree'' version of this statement: for every vertex, at most $\delta N$ of its out-neighbors precede it in the ordering, and at most $\delta N$ of its in-neighbors follow it in the ordering. At this point, most of the argument presented in the transitive case goes through: while there may be a small number of edges that we cannot use, the arguments are sufficiently robust to overcome this.

However, there is a major thing that stops working, which is instrumental to using the key claim above. Namely, if $v$ precedes $w$ in the ordering, and even if $v \to w$ in $T$, we may have that $vw$ does not receive color $i$, and yet not be able to conclude that $\ell_i(w)>\ell_i(v)$. Indeed, it is possible that all long  $i$-avoiding paths ending at $v$ have already passed through $w$, and hence we cannot simply extend such a path to end at $w$. This is a fundamental issue
which is not present in the fully transitive setting: there, a directed path must be monotone, and hence cannot have already used a ``future'' vertex. 

In order to overcome this issue, we split the set of colors into \emph{long} colors---those colors $i$ where there is an $i$-avoiding path of length at least $\gamma N$, for some small $\gamma>0$---and \emph{short} colors. For short colors, the arguments used above can now be made to work: the longest $i$-avoiding path ending at $v$ has length at most $\gamma N$, hence there are relatively few vertices $w$ for which we cannot argue about the color of the edge $vw$, as most vertices $w$ are not on such a path. Dealing with these ``exceptional'' vertices now requires some care, but no genuinely new ideas. On the other hand, we know essentially nothing about the long colors, except that they are long. If we only aimed to prove that $T$ contains a long color-avoiding directed path, then long colors would immediately help us; unfortunately, we strengthened the inductive hypothesis to argue about $\Pi(T)$, so we cannot simply say we are done if there is a long color.

In order to overcome this issue, the final idea is to \emph{weaken} the inductive hypothesis. Rather than inductively studying the function $\Pi(T)$ defined above, we will instead argue inductively about the modified function
\[
    \prod_{i=1}^q \min \{\ell_i(T), \gamma N\}.
\]
For the short colors, the term controlling the minimum is $\ell_i(T)$, hence we can use the same argument as before. For long colors, the main term is $\gamma N$, and luckily this quantity is trivial to handle inductively: for example, since $\gamma N=\gamma \frac N2 + \gamma \frac N2$, we can automatically obtain the doubling we seek in the dream scenario for long colors. With these modifications, the rest of the proof goes through essentially as before.

\section{Simple proofs}\label{sec:simple}
\subsection{Proof of Theorem \ref{labelfarfromtransitive}}

\begin{proof}[Proof of Theorem \ref{labelfarfromtransitive}]
The first two authors \cite[Lemma 1.3]{FoSu} proved the following quantitative variant of the triangle removal lemma in tournaments: There is a constant $c >0$ such that if a tournament on $N$ vertices is $\delta$-far from  transitive, then it contains at least $c\delta^2 N^3$ cyclic  triangles. For each cyclic triangle, arbitrarily label the edges cyclically as $1,2,3$. Since there are $q$ colors, there are $q^3$ possible colorings of the edges of a cyclic triangle with the edges labeled as $1,2,3$. By the pigeonhole principle, there are at least $c\delta^2 N^3/q^3$ cyclic triangles whose labeled edges have the same color pattern (call this pattern $P$). Greedily delete edges from the tournament that are in fewer than $2c\delta^2 N/q^3$ such cyclic triangles of color pattern $P$.  As there are fewer than $N^2/2$ edges, we still have more than $c\delta^2 N^3/q^3 - (N^2/2)(2c\delta^2 N/q^3)=0$ cyclic triangles of color pattern $P$ remaining after we are done with this process. Each edge that remains (and there is at least one since there is at least one cyclic triangle remaining) extends to at least $2c\delta^2 N/q^3$ cyclic triangles of color pattern $P$. Starting with the vertices $v_1,v_2,v_3$ of one remaining cyclic triangle of color pattern $P$ with $v_1 \rightarrow v_2$, we can greedily build the desired directed path (and furthermore the square of a path with the desired properties) by adding, after $v_{i-1},v_i$, a vertex $v_{i+1}$ not already on the path such that $v_{i-1},v_i,v_{i+1}$ forms a cyclic triangle of color pattern $P$ in what remains. We can necessarily find such a vertex $v_{i+1}$ as long as $2c\delta^2 N/q^3 > i-1$ since each edge on the path extends to more than $2c\delta^2 N/q^3$ cyclic triangles of color pattern $P$ in what remains. Hence, the path length we get in the end is at least $2c\delta^2 N/q^3$. 
\end{proof}

\subsection{General reductions and bounds}

One useful observation is that monotone paths behave very cleanly under lexicographic product. Recall that an \emph{ordered graph} is a graph equipped with a linear order on its vertex set. Given two ordered graphs $G_1,G_2$, we define their \emph{lexicographic product} $G_1 \otimes G_2$ by first duplicating each vertex of $G_2$ to $\ab{G_1}$ vertices, and replacing each edge of $G_2$ by a complete bipartite graph between these blowup sets. We then insert a copy of $G_1$ into each of these blowup sets, while otherwise maintaining the order. That is, each original vertex of $G_2$ now corresponds to an interval of length $\ab{G_1}$ in the new ordering. 
\begin{proposition}\label{prop:monotone path product}
    Let $G_1,G_2$ be ordered graphs, and let $L_1,L_2$ be the length of the longest monotone path in $G_1,G_2$, respectively. Then the length of the longest monotone path in $G_1 \otimes G_2$ is exactly $L_1L_2$.
\end{proposition}
\begin{proof}
    The lower bound is immediate: we can replace each vertex of the longest monotone path in $G_2$ by the longest monotone path in $G_1$, thus obtaining a monotone path in $G_1 \otimes G_2$ of length $L_1 L_2$. For the upper bound, fix any monotone path in $G_1 \otimes G_2$. Note that this path can visit at most $L_2$ of the blowup parts, since every time it moves between blowup parts, it needs to use an edge from the blowup of $G_2$. Thus, under the natural projection to $G_2$, it traverses a monotone path in $G_2$, hence it visits at most $L_2$ parts. But within each part, it simply traverses a monotone path in $G_1$, hence it uses at most $L_1$ vertices in each part. In total, the number of vertices visited is at most $L_1L_2$. 
\end{proof}

As an immediate corollary, we obtain the same result for paths with restricted colors.
Let $K_1$ and $K_2$ be $q$-edge-colored ordered complete graphs (possibly on different numbers of vertices). We define their \emph{lexicographic product} $K_1 \otimes K_2$ in the same way as above: we duplicate each vertex of $K_2$ to $\ab{K_1}$ vertices, replace each edge of $K_2$ by a complete bipartite graph of the same color, and insert a copy of $K_1$ into each of these blowup sets, while otherwise maintaining the order. 
If $S\subseteq [q]$ is a set of colors, we say that a path is \emph{$S$-colored} if it is colored only by colors in $S$. 
\begin{proposition}\label{prop:lex product}
    Let $K_1,K_2$ be two $q$-edge-colored ordered complete graphs, and let $S \subseteq[q]$ be an arbitrary set of colors. Let $L_1,L_2$ be the length of the longest monotone $S$-colored path in $K_1,K_2$, respectively. Then the length of the longest $S$-colored monotone path in $K_1 \otimes K_2$ is exactly $L_1 L_2$. 
\end{proposition}
\begin{proof}
    Let $G_1,G_2$ be the ordered graphs comprising all edges in $K_1,K_2$, respectively, colored by a color in $S$. Then $G_1 \otimes G_2$ is precisely the set of edges in $K_1 \otimes K_2$ receiving a color in $S$. Thus, the claimed result follows immediately from \cref{prop:monotone path product}.
\end{proof}
One consequence of \cref{prop:lex product} is the upper bound $g_{q,r}(N) \leq f_{q,r}(N) \leq N^{r/q}$ for all $N$ which are a power of $q$. Indeed, let $N=m^q$, and let $K_1,\dots,K_q$ be ordered cliques on $m$ vertices, where all edges in $K_i$ are given color $i$. By \cref{prop:lex product}, for every $S \subseteq [q]$, the length of the longest $S$-colored monotone path in $K = K_1\otimes \dots \otimes K_q$ is precisely $m^{\ab S}$; in particular, the longest monotone path in $K$ receiving at most $r$ colors has length exactly $m^r = (m^q)^{r/q}$, proving the bound $f_{q,r}(N) \leq N^{r/q}$ for $N=m^q$.

Another immediate corollary of \cref{prop:lex product} is that the function $f_{q,r}(N)$ is sub-multiplicative. 
\begin{corollary}\label{cor:submultiplicative}
    For all integers $N_1,N_2, q,r$, we have $f_{q,r}(N_1)f_{q,r}(N_2) \geq f_{q,r}(N_1N_2)$.
\end{corollary}
\begin{proof}
    Let $K_1,K_2$ be $q$-edge-colored ordered complete graphs on $N_1,N_2$ vertices, respectively, with the property that every monotone path in $K_i$ which uses at most $r$ colors has length at most $f_{q,r}(N_i)$. Let $K = K_1 \otimes K_2$, so that $f_{q,r}(N_1 N_2) \leq f_{q,r}(K)$. On the other hand, for every set $S \subseteq [q]$ of exactly $r$ colors, we know from \cref{prop:lex product} that the maximum length of an $S$-colored monotone path in $K$ is at most $f_{q,r}(K_1)f_{q,r}(K_2)$, implying that $f_{q,r}(N_1 N_2)\leq f_{q,r}(K) \leq f_{q,r}(K_1)f_{q,r}(K_2) = f_{q,r}(N_1)f_{q,r}(N_2)$, as claimed. 
\end{proof}
One simple consequence of \cref{cor:submultiplicative} is that the limiting behavior of $f_{q,r}(N)$ is as a power of $N$, as stated in the following proposition.
\begin{proposition}
    For every $q >r \geq 1$, there exists some $\alpha \in [0,1]$ such that $f_{q,r}(N)=N^{\alpha+o(1)}$ as $N \to \infty$.
\end{proposition}
\begin{proof}
    Fix $q >r\geq 1$, and define $s_k = \log (f_{q,r}(2^k))$ for every integer $k \geq 0$.
    Then the sequence $(s_k)_{k\geq 1}$ is sub-additive, since \cref{cor:submultiplicative} shows that
    \[
    s_{k_1+k_2} = \log (f_{q,r}(2^{k_1+k_2})) \leq \log (f_{q,r}(2^{k_1}) f_{q,r} (2^{k_2})) = \log (f_{q,r}(2^{k_1}))+ \log (f_{q,r}(2^{k_2})) = s_{k_1}+s_{k_2}
    \]
    for all integers $k_1,k_2\geq 0$. Recall that Fekete's lemma states that for every sub-additive sequence $(s_k)$, the limit $\lim_{k \to \infty} s_k/k$ exists (but may be $-\infty$). Let us denote this limit by $\alpha$. Note that since $f_{q,r}(N)\geq 1$ for all $N$, we have that $s_k \geq 0$ for all $k$, and therefore $\alpha\geq 0$. Similarly, since $f_{q,r}(N)\leq N$ for all $N$, we have $s_k \leq k$ for all $k$, and thus $\alpha \leq 1$.

    Finally, we note that $s_k/k=\alpha+o(1)$, or equivalently $s_k =\alpha k +o(k)$, or equivalently $f_{q,r}(2^k) = 2^{\alpha k + o(k)}=(2^k)^{\alpha+o(1)}$. This proves the claim in case $N$ is of the form $2^k$ for some integer $k$. For all other $N$, we argue as follows: if $k = \flo{\log N}$, so that $2^k \leq N < 2^{k+1}$, then we have $2^k = N^{1+o(1)}$ and $2^{k+1}=N^{1+o(1)}$. Since the function $f_{q,r}$ is monotone, we have
    \[
    N^{\alpha+o(1)} = (2^k)^{\alpha+o(1)} = f_{q,r}(2^k) \leq f_{q,r}(N) \leq f_{q,r}(2^{k+1}) = (2^{k+1})^{\alpha+o(1)} = N^{\alpha+o(1)}.\qedhere
    \]
\end{proof}

As another consequence of \cref{cor:submultiplicative}, we can prove a ``finitary'' version of this limiting statement, similar to the use of the tensor product trick in the study of Sidorenko's conjecture (e.g.\ \cite{MR1682960,MR2738996}). It shows that whenever we prove a lower bound of the form $f_{q,r}(N) \geq N^{\alpha-o(1)}$ for some $\alpha$ and some $o(1)$ term tending to zero as $N\to \infty$, we can automatically deduce the stronger bound $f_{q,r}(N) \geq N^\alpha$ for all $N$.
\begin{proposition}\label{prop:c=1}
    Suppose that for all $N$, we have the estimate $f_{q,r}(N) \geq N^{\alpha-\varepsilon(N)}$, for some function $\varepsilon(N)$ with $\lim_{N\to \infty}\varepsilon(N)=0$. Then, in fact, we have $f_{q,r}(N) \geq N^\alpha$ for all $N$.
\end{proposition}
\begin{proof}
    Suppose for contradiction that $f_{q,r}(N) < N^\alpha$ for some $N$. In particular, we may write $f_{q,r}(N)=N^{\alpha-\delta}$ for some constant $\delta>0$.  By \cref{cor:submultiplicative}, for every $t\geq 1$, we have
    \[
        f_{q,r}(N^t) \leq f_{q,r}(N)^t = (N^{\alpha-\delta})^t=(N^t)^{\alpha-\delta}.
    \]
    Now, we pick $t$ sufficiently large so that $\varepsilon(N^t)<\delta$, which contradicts our assumption that $f_{q,r}(N^t) \geq (N^t)^{\alpha-\varepsilon(N^t)}$. 
\end{proof}

Recall from \cref{sec:overview} that in our proof of \cref{thm:main}, we work inductively with the quantity
\[
    \Pi(T) \coloneqq \prod_{i=1}^q \ell_i(T),
\]
where $\ell_i(T)$ denotes the length of the longest $i$-avoiding directed path in $T$.
A priori, working with this product
may appear to be wasteful, since perhaps the longest color-avoiding path is much longer than the geometric mean of all of them. As an immediate consequence of \cref{prop:lex product}, we can show that in the setting of transitive tournaments, this is not the case, and that the two notions are equivalent for large $N$.
\begin{proposition}
    Let $K$ be an $N$-vertex $q$-edge-colored ordered complete graph. There is a $q$-edge-colored ordered clique $\wt K$ on $N^q$ vertices for which $\ell_i(\wt K)=\Pi(K)$ for all $i \in [q]$.
\end{proposition}
\begin{proof}
    Let $K_1,\dots,K_q$ be obtained from $K$ by cyclically permuting the colors $q$ times. That is, $K_1=K$, and $K_t$ is obtained from $K_{t-1}$ by recoloring each edge of color $i$ to color $i+1\pmod q$. Let $\wt K=K_1\otimes \dots \otimes K_q$, so that $\ab{\wt K}=N^q$. By \cref{prop:lex product}, we know that
    \[
        \ell_i(\wt K) = \prod_{t=1}^q \ell_i(K_t).
    \]
    But $K_t$ is obtained from $K$ by cyclically permuting the colors, hence $\prod_{t=1}^q \ell_i(K_t) = \prod_{i=1}^q \ell_i(K)=\Pi(K)$, as claimed. 
\end{proof}

Finally,
the following lemma records some simple relations between the quantities $f_{q,r}(N)$ and $g_{q,r}(N)$ for different choices of the parameters $r$ and $q$. 
\begin{lemma}\label{lem:merge colors}
    Let $q >r \geq 1$ and $N\geq 1$ be integers.
    \begin{enumerate}[label=(\arabic*)]
        \item We have that $f_{q,r}(N) \geq f_{q-t, r-t}(N)$ and $g_{q,r}(N) \geq g_{q-t,r-t}(N)$ for all $0 \leq t<r$.\label{it:merge t}
        \item We have $f_{q,r}(N) \geq f_{q/d,r/d}(N)$ and $g_{q,r}(N) \geq g_{q/d,r/d}(N)$ for all $d$ dividing $\gcd(q,r)$. \label{it:merge d}
        \item If $p = \flo{\frac{q}{q-r}}\geq 2$, we have $f_{q,r}(N) \geq f_{p,p-1}(N)$ and $g_{q,r}(N) \geq g_{p,p-1}(N)$.\label{it:take floor}
    \end{enumerate}
\end{lemma}
\begin{proof}
    Let $T$ be a $q$-edge-colored $N$-vertex tournament. For \ref{it:merge t}, arbitrarily select $t+1$ colors and merge them into a single color, thus obtaining a $(q-t)$-edge-colored tournament. Any directed path in this auxiliary coloring which receives at most $r-t$ colors corresponds to a directed path in $T$ receiving at most $r$ colors, proving that $g_{q,r}(N) \geq g_{q-t,r-t}(N)$. As this argument does not change the tournament, we can restrict it to transitive tournaments to obtain the claim for $f_{q,r}$. 

    For \ref{it:merge d}, we arbitrarily group the $q$ colors into $q/d$ sets of size $d$, and merge the colors in each set. We obtain a $(q/d)$-edge-colored tournament, in which every directed path receiving at most $r/d$ colors corresponds to a path in $T$ receiving at most $r$ colors. This shows $g_{q,r}(N)\geq g_{q/d,r/d}(N)$, and applying the same argument to transitive tournaments yields the claim for $f_{q,r}$.

    Finally, for \ref{it:take floor}, write $q = p(q-r)+t$ for some $0 \leq t<q-r$. Note that $p \geq 2$ implies $q \geq 2q-2r$, or equivalently $q-r \leq r$, so $0\leq t< r$. We also have $q-t=p(q-r)$ and $(p-1)(q-r)=p(q-r)-(q-r)=(q-t)-(q-r)=r-t$ by definition of $t$. Therefore,
    \[
        f_{q,r}(N) \geq f_{q-t,r-t}(N) = f_{p(q-r),(p-1)(q-r)}(N)\geq f_{p,p-1}(N),
    \]
    using \ref{it:merge t} in the first inequality and \ref{it:merge d} in the second. The exact same argument holds for $g_{q,r}$.
\end{proof}
Although these bounds are very simple, they are sufficient to determine $f_{q,r}$ and $g_{q,r}$ in certain special cases. For example, recall that the Chv\'atal--Gy\'arf\'as--Lehel theorem states that $f_{q,1}(N) \geq g_{q,1}(N)\geq N^{1/q}$; both bounds are tight if $N$ is a power of $q$. Using \cref{lem:merge colors}\ref{it:merge d}, we conclude that for all $d \geq 1$,
\[
    f_{qd,d}(N) \geq g_{qd,d}(N)\geq g_{q,1}(N) \geq N^{1/q},
\]
which is again tight for $N$ a power of $q$ by \cref{prop:lex product}.

\section{Proof of Theorem \ref{thm:main}}\label{sec:main proof}

Henceforth, whenever working with a $q$-coloring, we will always assume that the palette of colors is $[q]$. Given a $q$-colored tournament $T$ and a color $i \in [q]$, we let $\ell_i(T)$ denote the length of the longest directed path in $T$ which does not use the color $i$. Recall that we define the length of a path to be the number of vertices it comprises, so that $\ell_i(T)\geq 1$ for every $T$ and every $i \in [q]$. 

We define a parameter $\gamma$ by
\begin{equation}\label{eq:gamma def}
    \gamma = \frac{1}{2q\cdot 2^{\sqrt{\log q}}}.
\end{equation}
For an $N$-vertex $q$-edge-colored tournament $T$, we then define
\[
    m_i(T) \coloneqq \min\{\ell_i(T), \gamma N\}
\]
and 
\[
    \Pi(T) \coloneqq \prod_{i=1}^q m_i(T)= \prod_{i=1}^q \min\{\ell_i(T),\gamma N\}
\]
and finally let $\Pi_q(N)$ be the minimum value of $\Pi(T)$ over all $N$-vertex $q$-colored tournaments $T$. 

Our main technical result, which immediately implies \cref{thm:main}, is the following lower bound for $\Pi_q(N)$.
\begin{proposition}\label{prop:T main product}
    We have that $\Pi_q(N) \geq c_q N^{q-Cq/\sqrt{\log q}}$, where $C>0$ is an absolute constant and $c_q>0$ is a constant depending only on $q$.
\end{proposition}
Note that \cref{prop:T main product} immediately implies \cref{thm:main}. Indeed, fix some $q$-colored tournament $T$ on $N$ vertices. If $\ell_i(T)\geq \gamma N$ for some $i\in [q]$, then we are done since $\gamma N \geq c_q N^{1-C/\sqrt{\log q}}$, by picking $c_q$ to be smaller than $\gamma$. Hence we may assume that $\ell_i(T)<\gamma N$ for all $i$, meaning that $\Pi(T)=\prod_{i=1}^q \ell_i(T)$. Therefore, there exists some $i \in [q]$ for which $\ell_i(T) \geq \Pi(T)^{1/q} \geq \Pi_q(N)^{1/q}$. Hence $T$ contains a path of length $\Pi_q(N)^{1/q} \geq c_q N^{1-C/\sqrt{\log q}}$ which avoids color $i$, as claimed in \cref{thm:main}.

Before presenting the proof of \cref{prop:T main product}, we record one more tool that we will use, namely a simple degree cleaning lemma which allows us to pass from a nearly transitive tournament to an ordered graph with high minimum degree. We recall that an \emph{ordered graph} is a graph equipped with a linear ordering $\prec$ of its vertex set. We say that an ordered graph $G$ is \emph{consistent} with a tournament on the same vertex set if, for all $v\prec w$, we have that $vw \in E(G)$ if and only if $v \to w$ in $T$. Note that a tournament and an ordering of its vertices uniquely defines an ordered graph (with the same ordering) that is consistent with it; conversely, every ordered graph uniquely determines a tournament consistent with it. 
\begin{lemma}\label{lem:clean degrees}
    Let $0<\delta<\frac 12$, and let $T_0$ be an $N_0$-vertex tournament which is $\delta^2$-close to transitive. Then there exists an integer $N \geq (1-\delta)N_0$, an $N$-vertex subtournament $T$ of $T_0$, and an ordered $N$-vertex graph $G$ with the following properties. $G$ is consistent with $T$, and every vertex of $G$ has at most $4\delta N$ non-neighbors in $G$. 
\end{lemma}
\begin{proof}
    By assumption, we can reverse the orientation of at most $\delta^2 N_0^2$ edges in $T_0$ in order to obtain a transitive tournament. Let us label $V(T)$ as $w_1,\dots,w_{N_0}$ according to this transitive ordering. Thus, in $T$, at most $\delta^2 N_0^2$ edges are backwards with respect to this ordering, i.e.\ there are at most $\delta^2 N_0^2$ pairs $(w_a,w_b)$ with $a<b$ and $w_a \leftarrow w_b$. Let $H_0$ be the graph of these backwards edges, i.e.\ the graph with vertex set $\{w_1,\dots,w_{N_0}\}$, where we set $(w_a,w_b)$ an edge for $a<b$ if and only if $w_a\leftarrow w_b$ in $T$. By assumption, $e(H_0)\leq (\delta N_0)^2$. Therefore, $H_0$ has at most $\delta N_0$ vertices of degree greater than $2\delta N_0$. By deleting these vertices, we obtain a graph $H$ on $N \geq N_0-\delta N_0 =(1-\delta) N_0$ vertices with maximum degree at most $2\delta N_0 \leq 4\delta N$. To conclude, we set $G =\overline H$, and set $T$ to be the subtournament induced on the vertices of $G$.
\end{proof}

We are now ready to prove \cref{prop:T main product}, and hence also \cref{thm:main}.
\begin{proof}[Proof of \cref{prop:T main product}]
    Recall that $\Pi_q(N) \geq 1$ for all $N$ and $q$, hence the result is vacuously true if $q$ is fixed and $C$ is large. Thus, we can and will assume henceforth that $q$ is at least some large absolute constant. The proof now proceeds by induction on $N$ for each fixed $q$; again, by picking $c_q$ sufficiently small, we can make the result vacuously true if $N$ is small, hence we have the base case of our induction. 

    We begin by defining several parameters that we will need during the proof. We let 
    \begin{equation}\label{eq:T parameters}
        p = \frac{24q}{2^{\sqrt{\log q}}}, \qquad s =2\gamma N= \frac{N}{q\cdot 2^{\sqrt{\log q}}}, \qquad \delta=\frac \gamma 8 = \frac{1}{16 q\cdot 2^{\sqrt{\log q}}}.
    \end{equation}
    To help the reader keep track of these, we note that in the proof, we will pick out $q$ disjoint sets of exactly $s$ vertices, hence we want $s$ to be smaller, but not much smaller, than $N/q$. Additionally, we will eventually encounter a special set of $p$ colors, hence we want $p$ to be smaller, but not much smaller, than $q$. The precise quantities defined above are obtained by optimizing the final contributions of the various steps of the proof. It will also be useful to notice for the future that $4\delta N = s/4$.

    We now fix some $N_0$, and assume that we have proved the result for all smaller values of $N$. We also fix a $q$-edge-colored tournament $T_0$ on $N_0$ vertices, and we aim to show that $\Pi(T_0) \geq c_q N_0^{q-Cq/\sqrt{\log q}}$. We first split into cases depending on whether or not $T_0$ is $\delta^2$-close to transitive, where we recall that $\delta$ is defined in \eqref{eq:T parameters}. If it is not, then we apply \cref{labelfarfromtransitive}. We find that $T_0$ contains a directed path of length at least $c\delta^4 N_0/q^3$, where $c>0$ is an absolute constant, which is colored by at most $3$ colors. Without loss of generality, let us assume that these three colors are $1,2,3$. This means that for every $i \in [q]\setminus\{1,2,3\}$, we have $\ell_i(T_0) \geq c\delta^4 N_0/q^3$.  As a consequence,
    \begin{align*}
        \Pi(T_0) &=\prod_{i=1}^q m_i(T) \geq \prod_{i=4}^q \min\{\ell_i(T_0),\gamma N_0\} \geq \prod_{i=4}^q \min \left\{ \frac{c\delta^4}{q^3},\gamma \right\}\cdot N_0=\left( \min \left\{ \frac{c\delta^4}{q^3},\gamma \right\} \right)^{q-3} N_0^{q-3}.
    \end{align*}
    Since $\delta$ and $\gamma$ both depend only on $q$, the first term is a positive constant depending only on $q$, and in particular can be made larger than $c_q$ by choosing $c_q$ appropriately. Moreover, by picking $C$ appropriately, we can ensure that $Cq/\sqrt{\log q}\geq 3$ for all $q$, hence $N_0^{q-3}\geq N_0^{q-Cq/\sqrt{\log q}}$. Therefore, we conclude that $\Pi(T_0) \geq c_q N_0^{q-Cq/\sqrt{\log q}}$, as desired. 

    We may thus assume that $T_0$ is $\delta^2$-close to transitive. We now apply \cref{lem:clean degrees} to pass to a subtournament $T$ on $N \geq (1-\delta)N_0$ vertices, as well as an ordered graph $G$ consistent with $T$ in which every vertex has at most $4\delta N$ non-neighbors. As $T$ is a subtournament of $T_0$, we have $\Pi(T_0)\geq \Pi(T)$, and therefore it suffices to prove that $\Pi(T) \geq c_q N_0^{q-Cq/\sqrt{\log q}}$. Moreover, since $N \geq (1-\delta)N_0\geq 2^{-2\delta}N_0$, it in turn suffices to prove that 
    \begin{equation}\label{eq:goal}
    \Pi(T)\geq 2^{\delta q}\cdot c_q\cdot N^{q-Cq/\sqrt{\log q}},
    \end{equation}
    which will be our goal for the remainder of the proof. We recall that our inductive hypothesis is that for all $N'<N$, we have $\Pi_q(N') \geq c_q\cdot (N')^{q-Cq/\sqrt{\log q}}$.

    We say that a color $i\in [q]$ is \emph{long} if $\ell_i(T)\geq \gamma N$, and we say that it is \emph{short} otherwise. Note that for any color $i \in [q]$, we have 
    \[
        m_i(T) = \min\{\ell_i(T), \gamma N\} = 
        \begin{cases}
            \ell_i(T) & \text{if $i$ is short},\\
            \gamma N & \text{if $i$ is long}.
        \end{cases}
    \]
    The long colors play a special role in our argument, and must be dealt with separately. In most of what follows, we argue almost exclusively about the short colors. 

    As $G$ is an $N$-vertex ordered graph, we henceforth identify the vertices of $G$ with the interval $[N]$, so that the ordering $\prec$ of $G$ agrees with the standard ordering of $[N]$. We alternate between these two notations as convenient.
    We begin by partitioning $[N]=V(G)=V(T)$ into four intervals $A,B,C,D$, of lengths $\frac N2-4s, 4s, 4s, \frac N2-4s$, respectively, where we recall that $s$ was defined in \eqref{eq:T parameters}. That is, we have
    \[
        A = [1,\tfrac N2-4s], \qquad B = (4s,\tfrac N2], \qquad C = (\tfrac N2, \tfrac N2+4s], \qquad D = (\tfrac N2+4s,N].
    \]
    For a vertex $v \in A\cup B$ and a color $i\in [q]$, we denote by $\ell_i(^\to v)$ the length of the longest directed path in $T[A \cup B]$ which ends at the vertex $v$ and whose edges avoid the color $i$. Note that we are restricting only to those directed paths which are entirely contained in $A\cup B$. Similarly, for $v\in C\cup D$, we denote by $\ell_i(v^\to)$ the length of the longest $i$-avoiding directed path in $T[C\cup D]$ that starts at the vertex $v$.

    For every color $i \in [q]$, we let $X_i$ consist of the $s$ vertices in $A \cup B$ which have the largest values of $\ell_i({}^\to v)$. That is, for every vertex $v \in A \cup B$, we compute the integer $\ell_i(^\to v)$, we rank the vertices of $A \cup B$ according to these integers, and then let $X_i$ comprise the $s$ top-most vertices in this ranking. If there are ties in the ranking, we break them arbitrarily. Similarly, we define $Y_i$ to comprise the $s$ vertices in $C \cup D$ with the largest values of $\ell_i(v^\to)$. By definition, $\ab{X_i}=\ab{Y_i}=s$ for all $i$.

    Let $0\leq r\leq q$ be the number of long colors in $T$. We split the $q-r$ short colors into several further types, as follows.
    Let $i$ be a short color. We call the color $i$ \emph{left-condensed} if $\ab{X_i \cap B} \geq s/2$, and \emph{left-diffuse} if $\ab{X_i\cap B}<s/2$. Similarly, $i$ is \emph{right-condensed} if $\ab{Y_i\cap C}\geq s/2$, and \emph{right-diffuse} otherwise. Finally, we call color $i$ \emph{condensed} if it is both left-condensed and right-condensed. We stress that we only apply these terms to short colors; in particular, whenever we speak of a condensed or left-diffuse color in what follows, that color will always also be short. 
    We now split into cases depending on the number of condensed colors, recalling that there are exactly $r$ long colors and $q-r$ short colors. 

    \paragraph{\bf Case 1: There are at least $\bm{q-r-2p}$ condensed colors} 
    Let $\Lambda \subseteq [q]$ denote the set of long colors, and let $\Xi \subseteq [q] \setminus \Lambda$ denote the set of condensed colors.
    Recall that every vertex of $G$ has at most $4\delta N = s/4$ non-neighbors. As a consequence, for every $i \in \Xi$, the number of edges of $G$ between $X_i \cap B$ and $Y_i \cap C$ is at least
    \[
        \sum_{v \in X_i \cap B} \left(\ab{Y_i \cap C}-\frac s4\right) \geq \sum_{v \in X_i\cap B} \left( \frac s2-\frac s4 \right) =\frac s4 \ab{X_i \cap B} \geq \frac{s^2}{8},
    \]
    where we use that color $i$ is right-condensed in the first inequality, and that it is left-condensed in the final inequality.
    
    Let us call a color $i \in \Xi$ \emph{useless} if all edges of $G$ between $X_i \cap B$ and $Y_i \cap C$ are of color $i$, and \emph{useful} otherwise. By the previous computation, if $i$ is useless, then there are at least $s^2/8$ color-$i$ edges between $B$ and $C$. Since there are at most $\ab B \ab C = 16s^2$ edges of $G$ between $B$ and $C$, we conclude that the number of useless colors is at most $(16s^2)/(s^2/8)=128$. 

    By definition, for every useful color $i$, there is at least one edge $v_iw_i\in E(G)$, with $v_i \in X_i \cap B$ and $w_i \in Y_i \cap C$, which does not receive color $i$ (for if there were no such edge, then $i$ would be useless). As $G$ is consistent with $T$ and $B$ precedes $C$ in the ordering, we know that the edge $v_iw_i$ is directed as $v_i\to w_i$. These facts imply that $T$ contains an $i$-avoiding directed path of length $\ell_i(^\to v_i)+\ell_i(w_i^\to)$, obtained by taking the longest $i$-avoiding directed path in $T[A\cup B]$ ending at $v_i$, concatenating it with the edge $v_i\to w_i$, and then concatenating that with the longest $i$-avoiding directed path in $T[C\cup D]$ starting at $w_i$. This is indeed a directed path since $v_i \to w_i$ and since the two paths we are concatenating have disjoint vertex sets; additionally, since the edge $v_i w_i$ does not receive color $i$, this is indeed an $i$-avoiding directed path.

    Let $T'$ be the subtournament of $T$ induced on the vertex set $(A \cup B) \setminus \bigcup_{i=1}^q X_i$. Note that $T'$ has at least $N'=\frac N2 - sq$ vertices. Moreover, for the special vertex $v_i$ found above, we claim that $\ell_i(^\to v_i) \geq \ell_i(T')$. Indeed, the longest $i$-avoiding directed path in $T'$ is entirely contained in $V(T')\subseteq A\cup B$, and ends at some vertex $u \in V(T')$, which in particular does not belong to $X_i$, as $X_i$ is disjoint from $V(T')$. As $X_i$ comprises the $s$ vertices with the highest values of $\ell_i(^\to \bullet)$, we see that $\ell_i(^\to v_i) \geq \ell_i(^\to u)=\ell_i(T')$.  Similarly, if we let $T''$ be the subtournament induced on $(C \cup D) \setminus \bigcup_{i=1}^q Y_i$, then $\ell_i(w_i^\to) \geq \ell_i(T'')$ for each $i \in [q]$. 

    Now, for each useful color $i \in \Xi$, we have that
    \[
        \ell_i(T) \geq \ell_i(^\to v_i)+\ell_i(w_i^\to) \geq \ell_i(T')+\ell_i(T'') \geq 2\sqrt{\ell_i(T') \ell_i(T'')},
    \]
    where the final bound is the inequality of arithmetic and geometric means. Since every such color is short, we also know that $m_i(T)=\min\{\ell_i(T),\gamma N\}=\ell_i(T)$, hence
    \[
        m_i(T)=\ell_i(T)\geq 2 \sqrt{\ell_i(T')\ell_i(T'')}\geq 2\sqrt{m_i(T')m_i(T'')}
    \]
    for every useful $i \in \Xi$.
    Similarly, for every long color $i \in \Lambda$, we have that $m_i(T)=\min\{\ell_i(T),\gamma N\}=\gamma N$, hence
    \[
        m_i(T) = \gamma N \geq \gamma \ab{T'}+\gamma \ab{T''} \geq m_i(T')+m_i(T'') \geq 2\sqrt{m_i(T')m_i(T'')}.
    \]
    In other words, we have found that for at least $q-2p-128$ of the colors $i$ (namely the useful or long colors), the inequality $m_i(T) \geq 2 \sqrt{m_i(T')m_i(T'')}$ holds. The remaining $2p+128$ colors are short, hence satisfy $m_i(T)=\ell_i(T)$. Moreover, for these remaining colors, we trivially have the inequality
    \[
        m_i(T)=\ell_i(T) \geq \max\{\ell_i(T'),\ell_i(T'')\} \geq \sqrt{\ell_i(T')\ell_i(T'')}\geq \sqrt{m_i(T')m_i(T'')}.
    \]
    Putting this all together, we find that
    \[
        \Pi(T) = \prod_{i=1}^q m_i(T) \geq 2^{q-2p-128} \prod_{i=1}^q\sqrt{m_i(T')m_i(T'')} = 2^{q-2p-128} \sqrt{\Pi(T')\Pi(T'')}.
    \]
    Finally, recall that $T'$ and $T''$ are both $q$-edge-colored tournaments on at least $N'=\frac N2-sq$ vertices. By the definition of $\Pi_q(N')$, this shows that $\Pi(T'), \Pi(T'') \geq \Pi_q(N')$. In conclusion, we have that
    \[
        \Pi(T) \geq 2^{q-2p-128} \Pi_q(N').
    \]
    By the inductive hypothesis, we know that $\Pi_q(N')\geq c_q (N')^{q-Cq/\sqrt{\log q}}$.
    Additionally, by the definition of $s$ in \eqref{eq:T parameters}, we know that $N' = (\frac 12-\frac{1}{2^{\sqrt{\log q}}})N$. Since we may assume that $q$ is sufficiently large, we have that $2^{\sqrt{\log q}}\geq \sqrt{\log q}$, and thus $N'\geq (\frac 12-\frac{1}{{\sqrt{\log q}}})N$.  Therefore,
    \begin{align*}
        \frac{(N')^{q-Cq/\sqrt{\log q}}}{N^{q-Cq/\sqrt{\log q}}} &\geq \left( \frac 12-\frac{1}{{\sqrt{\log q}}} \right)^{q-Cq/\sqrt{\log q}} = 2^{-q+Cq/\sqrt{\log q}} \left( 1-\frac{2}{{\sqrt{\log q}}} \right)^{q-Cq/\sqrt{\log q}}.
    \end{align*}
    Using the inequality $1-x \geq 2^{-2x}$, valid for all $x \in [0,\frac 12]$ (and recalling that $q$ is at least a sufficiently large constant, so that we may apply this inequality), we have that
    \begin{align*}
        \left( 1-\frac{2}{{\sqrt{\log q}}} \right)^{q-Cq/\sqrt{\log q}} \geq\left( 1-\frac{2}{{\sqrt{\log q}}} \right)^q \geq 2^{-4q/\sqrt{\log q}}.
    \end{align*}
    Combining these bounds, we conclude that
    \begin{align*}
        \frac{\Pi(T)}{c_q N^{q-Cq/\sqrt{\log q}}} &\geq \frac{2^{q-2p-128} \Pi_q(N')}{c_q N^{q-Cq/\sqrt{\log q}}}\geq 2^{q-2p-128} \cdot \frac{(N')^{q-Cq/\sqrt{\log q}}}{N^{q-Cq/\sqrt{\log q}}}\\
        &\geq 2^{q-2p-128} \cdot 2^{-q+Cq/\sqrt{\log q}} \cdot 2^{-4q/\sqrt{\log q}}\\
        &= 2^{(C-4)q/\sqrt{\log q} -2p -128}.
    \end{align*}
    Finally, we recall from \eqref{eq:T parameters} that $p = 24q/2^{\sqrt{\log q}} \leq q/\sqrt{\log q}$, where the final inequality holds since $q$ is sufficiently large. Therefore, the exponent above is at least $(C-6)q/\sqrt{\log q}-128$. In particular, as $C$ is sufficiently large, we have that this exponent is at least $2q/{\sqrt{\log q}}\geq 2\delta q$ for all $q$. We conclude that $\Pi(T) \geq 2^{2\delta q}\cdot c_q\cdot N^{q-Cq/\sqrt{\log q}}$. 
    This is precisely the inequality \eqref{eq:goal} that we set out to prove, and thus this completes the proof in Case 1.

    \paragraph{\bf Case 2: There are fewer than $\bm{q-r-2p}$ condensed colors.} In this case, there are at least $2p$ non-condensed colors, implying that there are at least $p$ left-diffuse colors or at least $p$ right-diffuse colors. The two cases are handled identically (in fact they are equivalent upon reversing the orientation of $T$), so we may assume without loss of generality that there are at least $p$ left-diffuse colors. We let $\Delta_L\subseteq [q]$ be a set of exactly $p$ left-diffuse colors.

    The proof in this case relies repeatedly on the following simple, but remarkably useful, observation, which follows immediately from the definition of the set $X_i$. 
    \begin{claim}\label{claim:exceptional}
        Let $i$ be a short color, and let $v \in X_i$. There is a set $E_i(v) \subseteq A \cup B$ with $\ab{E_i(v)}< 2s$ such that the following holds for all $w \in (A \cup B) \setminus  E_i(v)$. The vertices $v$ and $w$ are adjacent in $G$, and if $v\to w$ in $T$, then the color of the edge $vw$ is $i$.
    \end{claim}
    Here, we think of the set $E_i(v)$ as a set of ``exceptional'' vertices: apart from this small set of vertices, every out-directed edge of $v$ to $A \cup B$ receives color $i$. 
    \begin{proof}
        Let $\ol N(v)$ be the non-neighborhood of $v$ in $G$.
        We fix an $i$-avoiding directed path $P$ in $T[A\cup B]$ ending at $v$ of length $\ell_i(^\to v)$. As the color $i$ is short, we must have $\ab{V(P)}=\ell_i(^\to v)<\gamma N$. Now let $E_i(v)=V(P) \cup \ol N(v) \cup X_i$; we claim that this definition satisfies the desired property. 

        First, we have that $\ab{E_i(v)}\leq \ab{V(P)}+\ab{\ol N(v)}+\ab{X_i} \leq \gamma N+4\delta N+s < 2s$. Additionally, as $\ol N(v)\subseteq E_i(v)$, we certainly have that $v$ is adjacent in $G$ to all vertices in $(A\cup B) \setminus E_i(v)$. For the final property, fix some $w \in (A\cup B) \setminus E_i(v)$ with $v\to w$. As $w \notin X_i$ (since $X_i \subseteq E_i(v)$), we have that
        $\ell_i(^\to v) \geq \ell_i(^\to w)$, for $X_i$ consists of the $s$ vertices with the largest value of $\ell_i(^\to \bullet)$. Moreover, as $w \notin E_i(v)$, we see that $w \notin V(P)$. Consider the path $P +w$, obtained from $P$ by adding the edge $v\to w$ as the final edge. This is indeed a directed path, since $v\to w$ and $w \notin V(P)$. Additionally, if the edge $vw$ does not receive color $i$, then $P+w$ is an $i$-avoiding directed path in $T[A \cup B]$ ending at $w$. Therefore its length is at most $\ell_i(^\to w)$, from which we find
        \[
            \ell_i(^\to v) \geq \ell_i(^\to w) \geq \ab{V(P+w)} = \ab{V(P)}+1 = \ell_i(^\to v)+1,
        \]
        a contradiction. Thus, the edge $vw$ must receive color $i$. 
    \end{proof}

    For $i \in \Delta_L$, let $U_i = X_i \cap A$. As the color $i$ is left-diffuse, we know that $\ab{X_i \cap B}<s/2$, hence $\ab{U_i}\geq s/2$. 
    The following immediate consequence of \cref{claim:exceptional} shows that the sets $\{U_i : i \in \Delta_L\}$ are pairwise disjoint.
    \begin{claim}\label{claim:T Ui disjoint}
        Let $i,j \in \Delta_L$ be left-diffuse colors with $i \neq j$. Then the sets $U_i$ and $U_j$ are disjoint.
    \end{claim}
    \begin{proof}
        Suppose for contradiction that there is some vertex $v \in U_i \cap U_j$. Recall that $\ab B=4s$, and note that
        \[
            \ab{E_i(v)} + \ab{E_j(v)} < 2s+2s = 4s = \ab B,
        \]
        by \cref{claim:exceptional}.
        This implies that there is some vertex $w \in B \setminus (E_i(v) \cup E_j(v))$. By \cref{claim:exceptional}, we have $vw \in E(G)$. Moreover, since $v \in U_i \subseteq A$ and $w \in B$, we know that $v \prec w$, hence must have that $v\to w$ in $T$. Thus, by \cref{claim:exceptional}, the edge $vw$ receives color $i$. 
        But for the exact same reason, it must receive color $j$, a contradiction.
    \end{proof}
    Let $U = \bigcup_{i \in \Delta_L}U_i$. 
    By definition, for every $v \in U$, we must have that $v \in U_i$ for some $i\in \Delta_L$; moreover, this $i$ is unique by \cref{claim:T Ui disjoint}. Thus, for a vertex $v \in U$, we denote by $\iota(v)$ the unique index $i\in \Delta_L$ for which $v \in U_{i}$.
    
    Thanks to \cref{claim:exceptional}, we know a great deal about the colors of edges within $U$: almost all such edges are colored according to the index $\iota$ of their left endpoint. In particular, almost all edges in $U$ are colored with colors from $\Delta_L$. 
    In most of the remainder of the proof, we will use this structure to show that for all colors $k \notin \Delta_L$, we can find a long $k$-avoiding directed path within $U$; this is plausible, since we know that the edges within $U$ that \emph{are} colored $k$ must be quite restricted.

    Our first goal in this direction is to construct an appropriate structure within $U$ that we will use to glue long $k$-avoiding paths together. The precise structure we use is given in the following claim. For a vertex $v$ and vertex subset $S$, we write $v \prec S$ if $v \prec w$ for all $w \in S$, and we write $S \prec v$ if $w \prec v$ for all $w \in S$. 
    \begin{claim}\label{claim:gluing structure}
        There exist vertices $v_1,\dots,v_{t+1} \in U$, as well as sets $S_1,\dots,S_{t}\subseteq U$, with the following properties.
        \begin{enumerate}[label=(\roman*)]
            \item We have $t=p/24$ and $\ab{S_a}\geq s$ for all $a \in [t]$.\label{it:sizes}
            \item We have $v_1 \prec S_1 \prec v_2 \prec S_2 \prec v_3 \prec \dots \prec v_{t} \prec S_{t}\prec v_{t+1}$.\label{it:order}
            \item $S_a$ is contained in the out-neighborhood of $v_a$ and in the in-neighborhood of $v_{a+1}$, for all $a \in [t]$. That is, the edges are directed as $v_a \to S_a \to v_{a+1}$. \label{it:direction}
            \item All edges from $v_a$ to $S_a$ and all edges from $S_a$ to $v_{a+1}$ are colored with a color from $\Delta_L$.\label{it:color}
        \end{enumerate}
    \end{claim}
    \begin{proof}
        We begin by setting $v_1$ to be the first vertex in $U$ according to $\prec$. We now recursively define $S_a$ and $v_{a+1}$, given the value of $v_a$, such that the desired properties hold. Having already defined $v_1$, we can start the recursion.

        So suppose the value of $v_a$ is given. If there are fewer than $12s$ vertices in $U$ which come after $v_a$ under $\prec$, then we stop the recursion. Otherwise, we let $I_a$ denote the next $4s$ vertices in $U$ after $v_a$, and let $J_a$ be the subsequent $8s$ vertices in $U$ after $I_a$. We finally define $I_a' = I_a \setminus E_{\iota(v_a)}(v_a)$. 
        Note that 
        \[
            4s = \ab{I_a} \geq \ab{I_a'} = \ab{I_a \setminus E_{\iota(v_a)}(v_a)} \geq \ab{I_a} - \ab{E_{\iota(v_a)}(v_a)} \geq 4s - 2s = 2s,
        \]
        that is, that $2s \leq \ab{I_a'} \leq 4s$.

        Next, consider an auxiliary bipartite graph $\Gamma$ between $I_a'$ and $J_a$, where we join $v \in I_a'$ to $w \in J_a$ if and only if $w \in E_{\iota(v)}(v)$. As $\ab{E_{\iota(v)}(v)}\leq 2s$ by \cref{claim:exceptional}, we see that in $\Gamma$, all vertices in $I_a'$ have degree at most $2s$. Thus, the number of edges in $\Gamma$ is at most $2s\ab{I_a'} \leq 8s^2$. As a consequence, there must exist some vertex $w \in J_a$ whose degree in $\Gamma$ is at most $(8s^2)/\ab{J_a}=s$. We let $v_{a+1}$ be such a choice of $w \in J_a$. Finally, we let $S_a\subseteq I_a'$ be the non-neighbors of $v_{a+1}$ in $\Gamma$. By construction, we have
        \[
            \ab{S_a} \geq \ab{I_a'}-s \geq 2s-s=s.
        \]
        We have thus finished defining $S_a$ and $v_{a+1}$, and it remains to verify the claimed properties. We have just shown that $\ab{S_a}\geq s$, as claimed in \ref{it:sizes}. As $v_a \prec I_a \prec J_a$, we certainly have $v_a \prec S_a \prec v_{a+1}$, as claimed in \ref{it:order}. As $S_a \subseteq I_a' \subseteq (A\cup B) \setminus E_{\iota(v_a)}(v_a)$, we conclude from \cref{claim:exceptional} that all edges from $v_a$ to $S_a$ are oriented as $v_a \to S_a$, and they all receive color $\iota(v_a)$, proving the first halves of \ref{it:direction} and \ref{it:color}. For the corresponding claims about edges from $S_a$ to $v_{a+1}$, fix some $v \in S_a$. By the choice of $v_{a+1}$, we know that $v_{a+1} \notin E_{\iota(v)}(v)$, hence we again conclude that the edge $vv_{a+1}$ is oriented as $v \to v_{a+1}$, and that it receives color $\iota(v)\in \Delta_L$, completing the proofs of \ref{it:direction} and \ref{it:color}.

        Note that at every step of this process, we remove at most $12s$ vertices from $U$, hence we can continue this so long as $a<\ab U/(12s)$. Also note that $U$ is the union of the $p$ disjoint sets $U_i$ with $i \in \Delta_L$, and each such set has size at least $s/2$, by the definition of a left-diffuse color. This implies that $\ab U \geq ps/2$, and hence we can continue the process at least to step $t$, where
        \[
            t = \frac{\ab U}{12 s} \geq \frac{ps}{24s} = \frac{p}{24}.\qedhere
        \]
    \end{proof}

    For each $a \in [t]$, let $T_a$ be the subtournament of $T$ induced on the vertex set $S_a$. The key claim to complete the proof is the following, which allows us to glue together $k$-avoiding paths in all of the tournaments $T_a$, for every color $k \notin \Delta_L$. 
    \begin{claim}\label{claim:T glue paths}
        Let $k\in [q] \setminus \Delta_L$ be a color. We have that
        \[
            m_k(T) \geq \sum_{a=1}^t m_k(T_a).
        \]
    \end{claim}
    \begin{proof}
        Let $P_a$ be a longest $k$-avoiding directed path in $T_a$, for all $a \in [t]$. We concatenate these paths into a longer path $P$ defined as $v_1P_1v_2P_2 \dots P_t v_{t+1}$. That is, we start at $v_1$, go to the first vertex of the directed path $P_1$, traverse $P_1$, go to $v_2$, and then continue in this fashion. By \cref{claim:gluing structure}, this gives us a directed path, since $v_a \to S_a \to v_{a+1}$ for all $a \in [t]$. Moreover, all edges $v_a \to S_a \to v_{a+1}$ are colored by colors in $\Delta_L$, by \cref{claim:gluing structure}, and in particular are not colored by the color $k$. Thus, the path $P$ avoids the color $k$, and its length is at least $\sum_{a=1}^t \ell_k(T_a)$; this proves that
        \[
            \ell_k(T) \geq \sum_{a=1}^t \ell_k(T_a).
        \]
        We now split into cases depending on whether $k$ is long or short. If $k$ is short, then $m_k(T)=\ell_k(T)$, hence
        \[
            m_k(T)=\ell_k(T) \geq \sum_{a=1}^t \ell_k(T_a) \geq \sum_{a=1}^t m_k(T_a).
        \]
        On the other hand, if $k$ is long, then
        \[
            m_k(T) = \gamma N \geq \sum_{a=1}^t \gamma \ab{T_a} \geq \sum_{a=1}^t m_k(T_a),
        \]
        where the first inequality uses that the sets $S_1,\dots,S_t$ are pairwise disjoint subsets of $[N]$. 
        In either case, we have the claimed bound.
    \end{proof}
    From \cref{claim:T glue paths} and the inequality of arithmetic and geometric means, we find that
    \[
        m_k(T) \geq \sum_{a=1}^t m_k(T_a) \geq t \left( \prod_{a=1}^t m_k(T_a) \right)^{1/t}
    \]
    for every color $k \in [q] \setminus \Delta_L$. On the other hand, for every color $k \in \Delta_L$, we trivially have
    \[
        m_k(T) \geq \max\{m_k(T_1),\dots,m_k(T_t)\} \geq \left( \prod_{a=1}^t m_k(T_a) \right)^{1/t}.
    \]
    Combining these bounds, we find that
    \begin{align*}
        \Pi(T) &= \prod_{k=1}^q m_k(T) =\prod_{k\in \Delta_L} m_k(T) \cdot \prod_{k\in [q] \setminus \Delta_L} m_k(T)\\
        &\geq \prod_{k\in \Delta_L} \left( \prod_{a=1}^t m_k(T_a) \right)^{1/t} \cdot \prod_{k\in [q]\setminus \Delta_L} t \left( \prod_{a=1}^t m_k(T_a) \right)^{1/t}\\
        &= t^{q-p} \left( \prod_{a=1}^t \prod_{k=1}^q m_k(T_a) \right)^{1/t}\\
        &=t^{q-p} \left( \prod_{a=1}^t \Pi(T_a) \right)^{1/t}\\
        &\geq t^{q-p}\cdot \Pi_q(s),
    \end{align*}
    where the final step uses that each $T_a$ is a $q$-edge-colored tournament on $\ab{S_a}\geq s$ vertices.

    By the induction hypothesis on $s$ vertices, we know that $\Pi_q(s) \geq c_q s^{q-Cq/\sqrt{\log q}}$. From the definition of $s$ in \eqref{eq:T parameters}, we have that
    \[
        \frac{s^{q-Cq/\sqrt{\log q}}}{N^{q-Cq/\sqrt{\log q}}} = \left( \frac{1}{q\cdot 2^{\sqrt{\log q}}} \right)^{q-Cq/\sqrt{\log q}} = q^{-q+Cq/\sqrt{\log q}} \cdot 2^{-q\sqrt{\log q}+Cq} \geq q^{-q} \cdot 2^{(C-1)q\sqrt{\log q}}.
    \]
    Similarly, since $t=p/24$ from \cref{claim:gluing structure}, and from the definition of $p$ from \eqref{eq:T parameters}, we have that
    \[
        t^{q-p} = \left( \frac{q}{2^{\sqrt{\log q}}} \right)^{q-p} =q^{q-p} \cdot 2^{-q \sqrt{\log q}+p\sqrt{\log q}} \geq q^{q-p} \cdot 2^{-q\sqrt{\log q}}.
    \]
    Since $q$ is sufficiently large, we have that $2^{\sqrt{\log q}}\geq 24\sqrt{\log q}$, and hence $p \leq q/\sqrt{\log q}$. Therefore,
    \[
        t^{q-p}\geq q^{q-p} \cdot 2^{-q\sqrt{\log q}} \geq q^{q-q/\sqrt{\log q}} \cdot 2^{-q\sqrt{\log q}}=q^q \cdot 2^{-2q\sqrt{\log q}}. 
    \]
    Combining these bounds, we find that
    \begin{align*}
        \frac{\Pi(T)}{c_q N^{q-Cq/\sqrt{\log q}}} &\geq \frac{t^{q-p}\cdot \Pi_q(s)}{c_q N^{q-Cq/\sqrt{\log q}}} \geq t^{q-p}\cdot \frac{s^{q-Cq/\sqrt{\log q}}}{N^{q-Cq/\sqrt{\log q}}}\\
        &\geq \left( q^q \cdot 2^{-2q \sqrt{\log q}} \right) \left( q^{-q} \cdot 2^{(C-1)q\sqrt{\log q}} \right)\\
        &= 2^{(C-3)q\sqrt{\log q}}. 
    \end{align*}
    In particular, as $C$ is sufficiently large, then this exponent is at least $2\delta q$, hence we find that $\Pi(T) \geq 2^{\delta q}\cdot c_q\cdot N^{q-Cq/\sqrt{\log q}}$, as we wanted to prove in \eqref{eq:goal}. This completes the proof of Case 2, and thus of \cref{prop:T main product}.
\end{proof}

\section{Concluding remarks}\label{sec:conclusion}
There are a number of natural questions left open by our work. First, recall that for $q \geq 4$, \cref{prop:F=G} implies that $F_{q,q-1}(n)=G_{q,q-1}(n)$. In other words, in the vector question, the added flexibility of choosing a $(q-1)$-comparable set (where the tournament of comparability may be arbitrary) does not add anything; the largest such set is actually a $(q-1)$-increasing sequence, i.e.\ whose tournament is transitive. In the setting of color-avoiding paths, \cref{labelfarfromtransitive} gives an approximate version of an analogous statement: a tournament all of whose color-avoiding paths are short must be $o(1)$-close to transitive. However, we do not know how to prove that the extremal construction is precisely transitive, and, in fact, do not expect this to be true.
\begin{problem}\label{conj:transitive not extremal}
    For some integers $q \geq 4$ and $N$, does there exist a non-transitive $q$-edge-colored $N$-vertex tournament whose longest color-avoiding path has strictly fewer than $f_{q,q-1}(N)$ vertices?
\end{problem}
Note that the answer to this problem is affirmative for $q=3$, since already the old construction of Hamaker and Stein \cite{MR754867} gives a $2$-comparable set of vectors in $[7]^3$ which has strictly more vectors than the longest $2$-increasing sequence in $[7]^3$, yielding such a non-transitive tournament. However, \cref{prop:F=G} shows that no such vector construction can exist for $q \geq 4$, hence we do not know how to resolve \cref{conj:transitive not extremal}.

Another natural question left open by our work is to extend our results to the more general setting of  directed paths in $q$-edge-colored tournaments which receive at most $r$ colors, for more general parameters $1 \leq r < q$. For example, \cref{labelfarfromtransitive} allows us to find long paths that receive only $3$ colors whenever we are working with a tournament that is far from transitive; however, many of our other arguments seem specialized to the color-avoiding setting of $r=q-1$, and there appear to be difficulties in extending them to the most general setting. On the other hand, the simple reductions recorded in \cref{lem:merge colors} do allow us to obtain some non-trivial results in this more general setting as a consequence of \cref{thm:main}. For example, we can prove the following result, which gives a good estimate when $q/r$ is close to $1$.
\begin{proposition}\label{prop:generalr}
    For every $\varepsilon>0$, there exists $\delta>0$ such that the following holds. If $q >r$ are integers with $r\geq (1-\delta)q$, then $g_{q,r}(N) \geq c_{q,r} N^{1-\varepsilon}$ and $f_{q,r}(N) \geq N^{1-\varepsilon}$ for all $N$, where $c_{q,r}>0$ is a constant depending only on $q$ and $r$.
\end{proposition}
\begin{proof}
    By \cref{thm:main}, there exists some $p_0=p_0(\varepsilon)\geq 2$ such that for all $p \geq p_0$, we have $g_{p,p-1}(N) \geq c_p N^{1-\varepsilon}$. Let $\delta = 1/p_0$, and fix integers $q>r$ with $r\geq (1-\delta)q$. Letting $p = \flo{\frac{q}{q-r}}$, we see that $p \geq \flo{\frac q{\delta q}} = p_0$. Therefore, \cref{lem:merge colors}\ref{it:take floor} implies $g_{q,r}(N) \geq g_{p,p-1}(N)\geq c_p N^{1-\varepsilon}$, which is the claimed result upon setting $c_{q,r}=c_p$ (noting that $p$ depends only on $q$ and $r$). Since $f_{q,r}(N)\geq g_{q,r}(N)$, we find that $f_{q,r}(N) \geq c_{q,r}N^{1-\varepsilon}$, which implies $f_{q,r}(N) \geq N^{1-\varepsilon}$ by \cref{prop:c=1}.
\end{proof}
We stress that these questions are also interesting in the transitive case. In particular, while \cref{prop:sawin-tao} provides good bounds on $F_{q,r}(n)$, once $r \neq q-1$, this problem is no longer equivalent to the study of $f_{q,r}(N)$, but rather to a set-colored Ramsey question. Thus, determining the behavior of $f_{q,r}(N)$ remains an interesting open problem.

\subsection*{Acknowledgments} We would like to thank Cosmin Pohoata for bringing the work of Sawin and Tao to our attention, and for helpful discussions on this topic.

\end{document}